\documentclass[11pt,twoside]{article}

\usepackage{amsmath,amsfonts,bm}

\def\1{\bm{1}}

\DeclareMathAlphabet{\mathsfit}{\encodingdefault}{\sfdefault}{m}{sl}
\SetMathAlphabet{\mathsfit}{bold}{\encodingdefault}{\sfdefault}{bx}{n}

\usepackage[utf8]{inputenc} 
\usepackage[T1]{fontenc}    
\usepackage{setspace}
\usepackage{booktabs}       
\usepackage{amsfonts}       
\usepackage{nicefrac}       
\usepackage{microtype}      

\usepackage{subfigure}
\usepackage{epsf}
\usepackage{epsfig}
\usepackage{fancyhdr}
\usepackage{graphics}
\usepackage{graphicx}
\usepackage{psfrag}
\usepackage{fullpage}
\usepackage{pdfpages}

\usepackage{natbib}

\usepackage{url}
\usepackage[colorlinks,linkcolor=magenta,citecolor=blue, pagebackref=true]{hyperref}
\renewcommand*{\backrefalt}[4]{%
    \ifcase #1 \footnotesize{(Not cited.)}%
    \or        \footnotesize{(Cited on page~#2.)}%
    \else      \footnotesize{(Cited on pages~#2.)}%
    \fi}
\usepackage{color}

\usepackage{amsthm}
\usepackage{amsmath}
\usepackage{amssymb,bbm}
\usepackage{caption}
\usepackage{algorithmic}
\usepackage{algorithm}
\usepackage{textcomp}
\usepackage{siunitx}
\usepackage{wrapfig}
\usepackage{algorithmic}
\usepackage{algorithm}
\usepackage{mathrsfs}  
\usepackage{multirow}
\usepackage{multicol}

\newcommand{\iid}{\overset{\textnormal{iid}}{\sim}}
\newcommand{\wc}{\overset{w}{\longrightarrow}}

\newcommand{\kl}{\textnormal{KL}}

\newtheorem{remark}{Remark}

\newtheorem{assumption}{Assumption}

\newtheorem{lemma}{Lemma}

\newtheorem{theorem}{Theorem}
\newtheorem{proposition}{Proposition}
\newtheorem{definition}{Definition}
\newtheorem{corollary}{Corollary}

\newcommand{\Wref}[1]{W\ref{#1}}

\newcommand{\kls}{\textnormal{KLS}}

\def\RR{\mathbb{R}}

\def\NN{\mathbb{N}}

\def\EE{\mathbb{E}}

\def\QQ{\mathbb{Q}}

\begin{document}
\onehalfspacing

\begin{center}

{\bf{\LARGE{Posterior Consistency in Parametric Models via a Tighter Notion of Identifiability}}}
  
\vspace*{.2in}
{\large{
\begin{tabular}{ccc}
Nicola Bariletto & Bernardo Flores & Stephen G. Walker
\end{tabular}
}}

\vspace*{.2in}

\begin{tabular}{c}
The University of Texas at Austin
\end{tabular}

\today

\vspace*{.2in}

\begin{abstract}
We investigate Bayesian posterior consistency in parametric density models with proper priors, challenging the common perception that the topic is settled. Classical results from the 1970s established posterior consistency as a consequence of MLE convergence, by combining regularity conditions with the assumption of model identifiability. In particular, the latter was treated as a background assumption and never examined in depth. This approach has gone largely unquestioned, partly due to a subsequent and nearly exclusive shift in focus to sieve-based methods tailored to nonparametric consistency. In our analysis, we place identifiability at the heart of posterior consistency. We show that, once one enlarges the model family to include weak limits, inconsistency fundamentally stems from a failure of identifiability at the true distribution. This finding reveals an important distinction: while such a failure can occur naturally in nonparametric models, it is highly implausible and essentially self-inflicted in parametric ones. This motivates a separate treatment of the two cases, with our focus here on the parametric side. Our theory leads to the finding that classical regularity assumptions are overly restrictive, while a simple tightening of identifiability suffices to establish posterior consistency even in irregular models where the MLE is inconsistent. Moreover, we prove that inconsistency requires the presence of densities with pathological oscillations that precisely match the true distribution. As we exemplify with an illustrative model, such behavior may arise only if the modeler possesses exact prior knowledge of the ground truth and adversarially encodes it in the model. Our example also underscores the need for distinct tools to study frequentist and Bayesian consistency: while MLE inconsistency stems from overfitting caused by likelihood peaks at the data—appropriately addressed through regularity or sieve methods—Bayesian parametric inconsistency is more naturally resolved by examining the identifiability structure of the model.
\end{abstract}

\end{center}

\section{Introduction}\label{sec:introduction}

Asymptotic consistency is a fundamental criterion for evaluating the quality of statistical estimation procedures. In Bayesian inference, the study of consistency of posterior distributions has been a highly active area of research, with the first contributions dating back to Joseph Doob's work \citep{doob1949application}. In his seminal result, Doob proved that, under a mild estimability condition, the posterior distribution asymptotically concentrates in arbitrary neighborhoods of the data-generating parameter value (under the assumption that a true one exists), for any parameter value belonging to a set of prior mass one. While remarkably parsimonious in its assumptions, Doob's approach was later reevaluated due to a major drawback: the theorem does not explicitly characterize the set of parameters at which the posterior is consistent, making it impossible to determine whether a given parameter belongs to this set, or how large its complement (where inconsistency may occur) is.

Subsequent major contributions appeared in the late 1960s and through the 1970s, focusing on parametric models for density estimation based on independently and identically distributed (iid) observations. A prominent example is the work of Andrew Walker \citep{walker1969} \citep[see also][]{berk1970consistency}, which leveraged the extensive classical theory on the consistency of frequentist procedures, particularly maximum likelihood estimators (MLEs), to establish posterior consistency. These approaches involved imposing identifiability and regularity conditions on the parametric family of likelihoods—such as smoothness of the log-likelihood ratio near zero and appropriate decay outside compact sets—to ensure both MLE consistency and concentration of the posterior around the MLE, thereby guaranteeing the convergence of posterior mass towards the true parameter value.

With Thomas Ferguson’s breakthrough definition of the Dirichlet process prior \citep{ferguson1973bayesian}, which made nonparametric inference practically feasible within the Bayesian framework, the literature on Bayesian consistency quickly shifted focus to infinite-dimensional models—partly due to the intriguing mathematical challenges they posed. Key early contributions include \cite{diaconis1986consistency, diaconis1986inconsistent}, followed by major developments in nonparametric density estimation \citep{ghosal1999, barron1999, walker2004squarerootsum}. The core idea in these works is to establish consistency using sieve methods which ensure that the posterior distribution concentrates within suitably defined  neighborhoods of the true data-generating density. These approaches also paved the way for important theoretical advances—such as the study of contraction rates \citep{ghosal2000convergencerates, ghosal2007convergence, lijoi2007convergence} and misspecified models \citep{kleijn2006misspecification, deblasi2013bayesian}—as well as for applications to specific nonparametric models \citep{lijoi2005consistency, ghosal2001convergence, ghosal2007posteriorDirichlet, ghosal2006posteriorGP}.

This trajectory of the literature has effectively shifted focus away from the parametric setting, with the study of nonparametric consistency nearly halting progress on the parametric side. As a result, conditions dating back more than 50 years, such as those in \cite{walker1969}, remain the standard tools for establishing consistency in finite-dimensional models.\footnote{See, e.g., the recent textbook treatment in \cite{ghosal2017fundamentals}, sect.~6.4. See instead \cite{mao2024calibrating, rustand2023bayesian, miller2021asymptotic, dougan2021bayesian} for very recent applied and methodological work citing \cite{walker1969}.} Indeed, to the best of our knowledge, although posterior consistency has been extensively studied for certain specific parametric models in recent years,\footnote{For instance, in the case of finite mixtures \citep{rousseau2011asymptotic, guha2021posterior}} there has been little to no major development in the realm of general parametric consistency since the original contributions from the early 1970s. In this article, we argue that this has resulted in a narrow view of parametric consistency and left key aspects of the topic underexplored. Moreover, these trends have limited the available conditions for establishing consistency to those developed decades ago—a time when, due to the lack of tools for a genuinely Bayesian analysis, the best available approach was to tie posterior consistency to the regular behavior of the MLE, thereby evaluating the asymptotic performance of Bayesian procedures through their alignment with frequentist ones.

\subsection{A new understanding of parametric consistency}\label{sub:new_perspective}

The key aim of this article is to conduct a fundamental reexamination of posterior parametric consistency. To set the stage for our discussion, consider the following standard modeling framework: let the sample space be $(\RR, \mathscr{B}(\RR))$,\footnote{Throughout the article, we only consider $\RR$ as our sample space to facilitate exposition, though much of our treatment easily extends to higher-dimensional scenarios.} where $\mathscr{B}(T)$ denotes the Borel $\sigma$-algebra on a topological space $T$, and let $\mathrm{d}x$ denote the Lebesgue measure on the real line. Define the statistical model $\mathcal{F}_\Theta := \{f_\theta : \theta \in \Theta\}$, where $f_\theta := \mathrm{d}F_\theta / \mathrm{d}x$ is the density (Radon--Nikodym derivative) associated with a probability measure\footnote{With a slight abuse of notation, we identify any probability measure $F$ on $(\RR, \mathscr B(\RR))$ with its cumulative disribution function (CDF), writing $F(x)\equiv F((-\infty, x])$ for all $x\in\RR$.} $F_\theta(\mathrm{d}x) \ll \mathrm{d}x$ parametrized by $\theta \in \Theta \subseteq \RR^p$, for some $p \in \NN$, and where $\Theta$ is assumed to be closed. As is natural for any sensibly parametrized family $\mathcal{F}_\Theta$ intended for density estimation, we assume that convergence of parameters in $\Theta$ is equivalent to convergence of the associated densities with respect to some strong metric, such as the Hellinger distance $d_h$; that is, for all $\theta, \theta_1, \theta_2, \ldots \in \Theta$,
\begin{equation*}
    \lim_{k \to \infty} \Vert \theta_k - \theta \Vert = 0 \iff \lim_{k \to \infty} d_h(f_{\theta_k}, f_\theta) = 0;
\end{equation*}
see Section \ref{sec:notation} for a formal definition of $d_h$. This allows us to identify $(\Theta, \Vert\cdot\Vert)$ with $(\mathcal F_\Theta, d_h)$ in terms of their metric structure, and to think of sequences of parameters $\theta_1, \theta_2,\ldots\in\Theta$ as sequences of densities in the model class $\mathcal F_\Theta$.

Now assume we observe a sample $X_{1:n} := (X_1, \dots, X_n) \iid F_{\theta_\star}$, where $\theta_\star \in \Theta$ is a fixed but unknown parameter to be estimated from the data. Following standard Bayesian procedures, a prior distribution $\Pi(\mathrm{d}\theta)$ on $(\Theta, \mathscr{B}(\Theta))$ gives rise, via Bayes' rule, to the posterior distribution
\begin{equation*}
    \Pi(\mathrm{d}\theta \mid X_{1:n}) = \frac{\prod_{i=1}^n f_\theta(X_i) \, \Pi(\mathrm{d}\theta)}{\int_\Theta \prod_{i=1}^n f_{\theta'}(X_i) \, \Pi(\mathrm{d}\theta')}.
\end{equation*}
Posterior consistency at $\theta_\star$ is then formulated as the requirement that
\begin{equation}\label{eq:posterior_consistency}
    \lim_{n \to \infty} \Pi(A_\varepsilon^c \mid X_{1:n}) = 0 \quad \text{a.s.-}F_{\theta_\star}^\infty
\end{equation}
for all $\varepsilon > 0$, where $A_\varepsilon := \{\theta \in \Theta : \Vert \theta - \theta_\star \Vert < \varepsilon\}$. Given the equivalence between the Euclidean and Hellinger metrics on \(\mathcal{F}_\Theta\), consistency implies that the posterior increasingly concentrates on densities that are arbitrarily close to the true density with respect to the Hellinger distance. Recalling that $d_h$ and the \(L^1\) metric are themselves equivalent on $\mathcal F_\Theta$, this provides a meaningful notion of convergence for densities—namely, pointwise closeness integrated with respect to the Lebesgue measure.

With this modeling framework in mind, our starting point is the seminal result by Lorraine Schwartz \citep[see also Theorem~\ref{thm:schwartz} below]{schwartz1964consistency}, who proved that if $\theta_\star$ lies in the Kullback-Leibler (KL) support of $\Pi$ (see Definition~\ref{def:KL_support} below), then the posterior distribution is \emph{weakly consistent} at $F_{\theta_\star}$, that is, \eqref{eq:posterior_consistency} holds when $A_\varepsilon$ is replaced with a neighborhood of $F_{\theta_\star}$ in the weak topology on $\mathcal F_\Theta$, metrized, for instance, by the L\'evy-Prokhorov metric (see Section~\ref{sec:notation} for a formal definition). In other words, Schwartz's theorem guarantees that, under a natural prior support condition, the posterior distribution concentrates in any small neighborhood of the true CDF. This result is remarkable in that it requires only a well-specified prior and applies even in nonparametric models—i.e., when $\Theta$ is infinite-dimensional rather than Euclidean. However, this type of consistency is not sufficient to characterize the posterior's behavior with respect to neighborhoods of the true density, rather than of the true CDF. The issue is that the posterior may concentrate on regions of the parameter space that yield arbitrarily good approximations of the true distribution in the weak sense, while still being far from the true density itself. A more detailed exploration of this phenomenon is provided in Section~\ref{sec:seq_identifiability}, but intuition can already be gained from Figure~\ref{fig:kde_oscillations}, which shows how a sequence of densities with increasing oscillations can converge in distribution to a target CDF while failing to converge in any meaningful to a proper density function.

\begin{figure}
    \centering
    \includegraphics[width=0.75\linewidth]{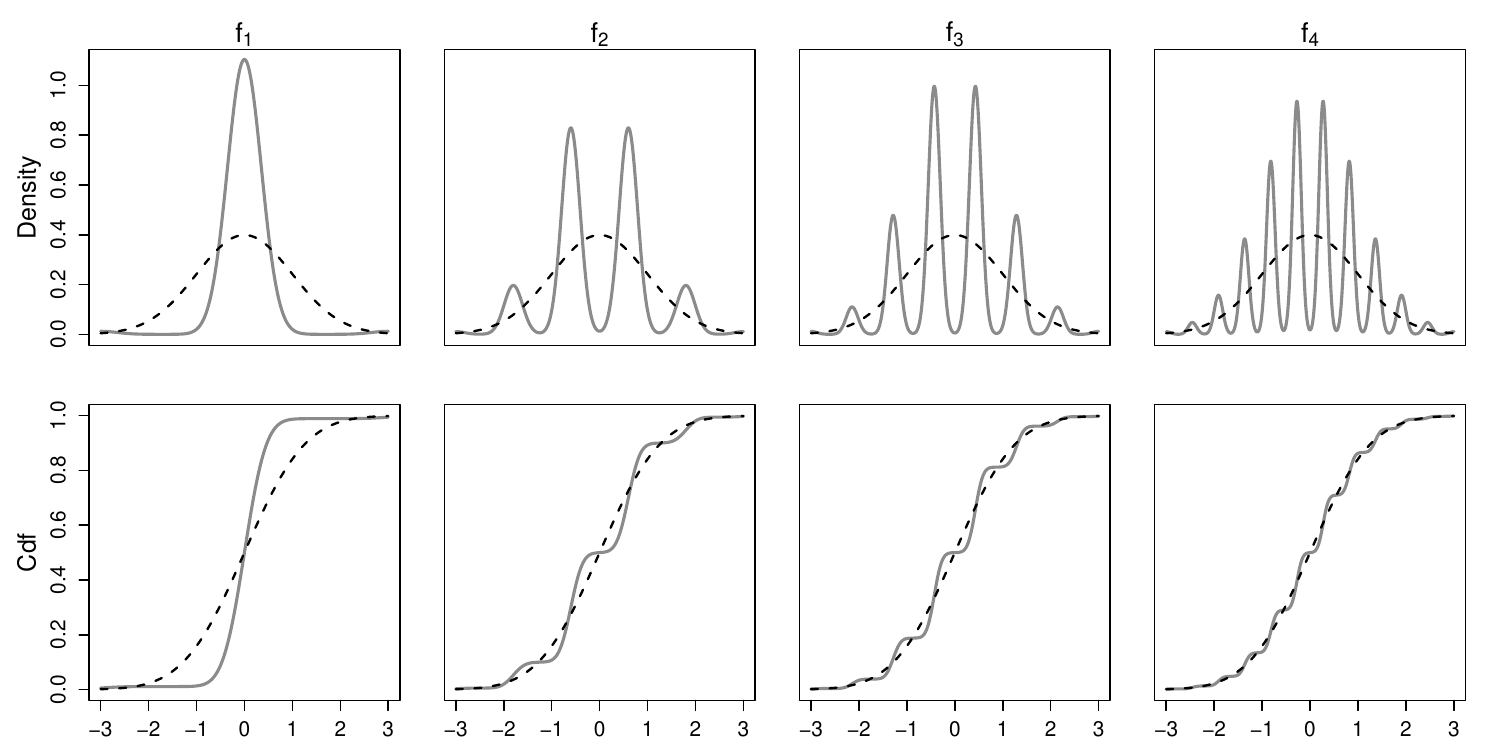}
    \caption{The initial segment of a sequence $f_1, f_2, \dots$ of indefinitely oscillating densities (solid gray, top row), whose corresponding sequence of CDFs (solid gray, bottom row) converges to a proper limiting distribution (dashed black). Such convergence is possible because the density sequence exhibits increasingly frequent oscillations around the density associated with the limiting CDF.}
    \label{fig:kde_oscillations}
\end{figure}

Interestingly, the strength of Schwartz’s result and of the proof techniques it introduced, based on exponential hypothesis tests, had a significant influence on subsequent works in nonparametric consistency. In particular, Schwartz's approach was later recognized \citep{ghosal2017fundamentals} as central to the sieve-based strategies of \cite{barron1999, ghosal1999}, which rely on controlling the complexity (e.g., the $L^1$ metric entropy) of certain sets of densities on which most of the prior mass is placed. In contrast, Schwartz’s result has received little attention in the parametric literature, likely due to the perception that its full potential had already been realized in the nonparametric setting. As a consequence, consistency results for parametric models have either relied on the classical conditions from the 1970s, such as those of \cite{walker1969}, or have been treated as simple corollaries of more abstract nonparametric analyses involving tools like sieves and entropy bounds.

However, we contend that these historical developments were shaped by a flawed assumption: that Bayesian parametric and nonparametric models for density estimation can be meaningfully analyzed through the same lens. Our disagreement with this assumption stems from the following observation. Schwartz’s result implies that the posterior asymptotically concentrates on parameter values that yield the true distribution function. Therefore, inconsistency in terms of densities (or parameter values) can occur only if there are regions of the parameter space, away from the true parameter $\theta_\star$, that produce CDFs arbitrarily close to $F_{\theta_\star}$ (and the posterior happens to concentrate around those instead of $\theta_\star$ itself). In other words, once the parameter space is enlarged to include all weak limits of sequences of distributions in the model, inconsistency can arise only if the true CDF corresponds to at least two distinct regions of the parameter space: a ``correct'' one (around $\theta_\star$) and at least one additional ``incorrect'' one (away from $\theta_\star$) that is obtained as the weak limit of some sequence of CDFs in the model. This situation, then, is best understood as a problem of identifiability at $\theta_\star$ in the enlarged parameter space, and if one is able to rule out the lack of identifiability described above, then consistency in terms of CDFs, guaranteed by Schwartz’s theorem, directly implies consistency in terms of parameters.

Once the key issue in achieving posterior consistency is recognized to be this extended notion of identifiability, which we later term \emph{sequential identifiability} (see Definition~\ref{def:sequential_identifiability} below), a fundamental discrepancy between parametric and nonparametric models becomes apparent. Nonparametric models, by their very nature, are inherently prone to a lack of sequential identifiability over large portions of the parameter space. Consider, for example, infinite Gaussian mixtures, which can weakly approximate any continuous distribution function while failing to converge in density due to excessively large values of the precision parameter. In such cases, the assumption of sequential identifiability no longer holds unless the prior support is appropriately constrained, and a nuanced analysis is required to account for the oscillatory behavior that gives rise to unidentifiability (see Sections~\ref{sec:seq_identifiability} and \ref{sec:discussion} for further elaboration).

On the other hand, the finite-dimensional nature of parametric models calls for a fundamentally different approach to sequential identifiability. For any parametric family that is identifiable in the usual sense, the inherently limited approximation power of the finite-dimensional parameter space—relative to the unbounded flexibility we will show to be required to induce sequential unidentifiability—makes it effectively impossible to weakly approximate more than a handful of distribution away from the corresponding parameter values. As we demonstrate in Section~\ref{sec:counterexample}, even a highly oscillatory family of densities, like the one we design for illustration, may fail sequential identifiability at most at a few isolated parameter values (just a single one, in our example). More generally, Theorem~\ref{thm:oscillations} shows that for sequential unidentifiability to occur at the true parameter, the model must include densities with arbitrarily frequent, finely tuned oscillations that integrate exactly to the true distribution function. Such behavior is highly implausible in any practical modeling context: because a parametric model allows, at most, for a small number of sequentially unidentifiable parameter values, posterior inconsistency would require exact prior knowledge of the true density in order to engineer the oscillations to align precisely with one of these few values, making inconsistency an unrealistic concern when dealing with parametric families genuinely designed for statistical modeling.\footnote{We note that \cite{walker2005data} identified the phenomenon of data tracking, arising from oscillatory densities, as a potential source of inconsistency. Our treatment, however, comprehensively reframes the issue in terms of identifiability and precisely characterizes the oscillatory patterns required for inconsistency to arise, showing that such behavior is not merely one possible cause, but in fact the sole mechanism (implausible in parametric models) through which inconsistency can occur.}

To summarize, the key distinction between parametric and nonparametric models lies in how they satisfy Schwartz’s theorem. In nonparametric models, the posterior can concentrate on densities that flexibly oscillate around (nearly) any candidate data-generating process, allowing them to weakly approximate the true distribution even when far from the true density—in our terminology, sequential unidentifiability holds over the whole (or most of) the model space. In contrast, parametric models generally lack this flexibility, so for Schwartz’s result to apply, the posterior must concentrate at the true parameter value.

Our analysis also highlights how, in a certain sense, the classical literature on parametric posterior consistency followed the MLE consistency literature in the ``wrong'' direction. Indeed, the theory of MLE consistency begins with the standard assumption of identifiability and adds suitable regularity conditions to ensure convergence. Parametric posterior consistency was then made to follow MLE consistency down the path of regularity,\footnote{One may argue that this was also true in the nonparametric setting, where the sieve conditions used to prove posterior consistency closely mirror those used for sieve MLE convergence \citep[e.g., see][]{wong1995probability, shen1994convergence, vandegeer2000empirical}.} at a time when the literature had not yet recognized that, in light of Schwartz's theorem, identifiability—not regularity—is the truly fundamental requirement. A central contribution of our work is to bring identifiability back to the core of the analysis of posterior consistency, yielding substantially better conditions in the parametric setting. Moreover, as the example in Section~\ref{sec:counterexample} illustrates, our strengthened notion of identifiability is specifically tailored to Bayesian procedures and has no bearing on MLE consistency, which may still fail due to the likelihood’s tendency to form peaks at the data and despite sequential identifiability at the data-generating parameter.

In summary, treating the parametric case separately from the nonparametric one, while still adopting Schwartz’s weak consistency result as a common foundation, naturally brings sequential identifiability to the forefront as the most natural and minimal condition for consistency in parametric models. As we show throughout the paper, this revised perspective allows us to go beyond the classical regularity assumptions from earlier works, enabling consistency even in irregular models where the maximum likelihood estimator performs poorly, such as our cosine-based example. This possibility reflects the different mechanisms behind inconsistency of frequentist and Bayesian procedures. For maximum likelihood, inconsistency often results from the likelihood overfitting the data, away from the true parameter value, by placing peaks at the observed points. Regularity or sieve conditions are appropriately designed to prevent this behavior, as seen in the classical assumptions discussed in Section~\ref{sec:classical_conditions}, particularly assumption~\Wref{walker_ass3}, as well as in our example in Section~\ref{sec:counterexample}, where these conditions fail because the model's oscillatory behavior allows each data point to be placed at a likelihood peak as the parameter diverges. By contrast, Schwartz's theorem implies that Bayesian posteriors are naturally able to recover the true distribution function, making sequential identifiability alone sufficient to ensure posterior consistency even in cases where the MLE is inconsistent.

\paragraph{Layout of the paper.} The rest of the article is structured as follows. After establishing basic notation and definitions in Section~\ref{sec:notation}, Section~\ref{sec:classical_conditions} reviews in detail the conditions proposed in the classical literature on posterior parametric consistency. In Section~\ref{sec:seq_identifiability}, we formally introduce our foundational approach based on sequential identifiability, providing general sufficient conditions for posterior consistency and linking inconsistency to a peculiar oscillatory behavior of the densities in the model. Section~\ref{sec:counterexample} illustrates our framework through a one-dimensional parametric model which, despite violating classical regularity assumptions, is readily handled using our methodology. This example also serves to highlight key aspects of our theoretical analysis, particularly in relation to oscillations. Section~\ref{sec:discussion} concludes the paper, while the proofs of all the theoretical results can be found in the supplementary material at the end of the article.

\section{Notation and basic definitions}\label{sec:notation}

Throughout, we use capital letters such as \(F\), \(G\), etc.\ to denote probability distributions, and lowercase letters \(f\), \(g\), etc.\ to denote the corresponding densities with respect to the Lebesgue measure. The standard Euclidean norm is denoted by \(\Vert \cdot \Vert\). We write \(\lesssim\) to indicate inequality up to a constant that, unless otherwise specified, is understood to be universal.

The Hellinger distance and the KL
divergence between two densities \(f\) and \(g\) are defined respectively as
\begin{align*}
    d_h(f, g) & := \left[\int_\RR \left(\sqrt{f(x)} - \sqrt{g(x)}\right)^2 \mathrm{d}x\right]^{1/2}, \\
    \kl(f,g) & := \int_\RR \ln\left(\frac{f(x)}{g(x)}\right) f(x)\,\mathrm dx,
\end{align*}
while the L\'evy–Prokhorov metric between two probability measures \(F\) and \(G\) is given by
\begin{equation*}
    d_w(F, G) := \inf\left\{\delta > 0:\, \forall A \in \mathscr B(\RR),\, F(A) \leq G(A^\delta) + \delta, \, G(A) \leq F(A^\delta) + \delta\right\},
\end{equation*}
where \(A^\delta := \{x \in \RR : \exists y \in A,\, \vert x - y \vert < \delta\}\) denotes the \(\delta\)-enlargement of the set \(A\). This distance metrizes weak convergence, which is denoted as $\wc$.

The next two definitions are fundamental. In particular, the KL support condition is central to Schwartz’s theorem (see Theorem~\ref{thm:schwartz} below) and forms the starting point of our analysis. As for identifiability \citep{casella2024statistical}, we assume without further mention that all considered statistical models satisfy this basic property at all parameter values.

\begin{definition}\label{def:KL_support}
    A parameter \(\theta_\star \in \Theta\) is said to be in the KL support of the prior \(\Pi\), denoted \(\theta_\star \in \textnormal{KLS}(\Pi)\), if
    \[
    \Pi\left(\left\{\theta \in \Theta : \kl(f_{\theta_\star}, f_\theta) < \varepsilon\right\}\right) > 0 \quad \text{for all } \varepsilon > 0.
    \]
\end{definition}

\begin{definition}\label{def:identifiability}
    The parametric family \(\mathcal F_\Theta := \{f_\theta : \theta \in \Theta\}\) is identifiable at \(\theta_\star \in \Theta\) if \(F_\theta \neq F_{\theta_\star}\) for all \(\theta \neq \theta_\star\).
\end{definition}

Finally, as noted in the introduction, we assume that the Euclidean and Hellinger metrics are equivalent on any parametric model under consideration. Accordingly, we will use the notions of Hellinger and Euclidean consistency interchangeably throughout the paper, without further mention. In particular, Hellinger consistency is to be understood as in \eqref{eq:posterior_consistency}, with the Euclidean metric replaced by $d_h$ in the definition of the neighborhood $A_\varepsilon$.

\section{Classical conditions for parametric posterior consistency}\label{sec:classical_conditions}

As discussed in Section~\ref{sec:introduction}, the predominant approach for establishing posterior consistency in parametric models is to verify that regularity conditions of the type proposed by \cite{walker1969} are satisfied. In particular, restricting attention to the one-dimensional setting \(\Theta \subseteq \RR\), in addition to the basic requirements that \(\Theta\) be closed and \(\mathcal F_\Theta\) identifiable, \cite{walker1969} required the following:

\begin{enumerate}
    \renewcommand{\labelenumi}{W\arabic{enumi}.}
    \item \label{walker_ass1}
    The set $\mathcal O_\theta := \{x \in \RR : f_\theta(x) = 0\}$ is the same for all $\theta \in \Theta$.
    
    \item \label{walker_ass2}
    For all $\theta \in \Theta$ and $x \in \RR$, there exists a function $H_\delta(x, \theta)$ such that, for all $\delta > 0$ small enough and all $\theta' \in \RR$ with $|\theta - \theta'| < \delta$, the following conditions hold:
    \begin{align*}
        |\ln f_\theta(x) - \ln f_{\theta'}(x)| &< H_\delta(x, \theta), \\
        \lim_{\delta \to 0} H_\delta(x, \theta) &= 0, \\
        \lim_{\delta \to 0} \int_\RR H_\delta(x, \theta) f_{\theta_\star}(x)\, \mathrm dx &= 0, \quad \forall \theta_\star \in \Theta;
    \end{align*}
    
    \item \label{walker_ass3}
    If $\Theta$ is unbounded, then for all $\theta_\star \in \Theta$ and sufficiently large $M > 0$, there exists a function $K_M(x, \theta_\star)$ such that
    \begin{equation*}
        \ln f_\theta(x) - \ln f_{\theta_\star}(x) < K_M(x, \theta_\star)
    \end{equation*}
    for all $|\theta| > M$, with
    \begin{equation*}
        \lim_{M \to \infty} \int_\RR K_M(x, \theta_\star) f_{\theta_\star}(x)\, \mathrm dx < 0.
    \end{equation*}
\end{enumerate}

Intuitively, assumption~\Wref{walker_ass1} ensures that the support of the model remains fixed across parameter values, thereby preventing singularities in likelihood ratios. Assumption~\Wref{walker_ass2} imposes a form of local uniform continuity on the log-likelihood function with respect to the parameter, captured via a modulus of continuity that vanishes both pointwise and in expectation under any true model. Finally, assumption~\Wref{walker_ass3} provides control over the relative tail behavior of the log-likelihood, requiring that log-likelihood ratios decay sufficiently fast---in an integrable manner---as the parameter value diverges, avoiding arbitrary peaks of the likelihood that may cause the MLE to overfit the data.

\cite{walker1969} showed that assumptions \Wref{walker_ass1}--\Wref{walker_ass3}, together with a positive and continuous prior density, imply posterior consistency at any $\theta_\star \in \Theta$ as a consequence of MLE consistency. While technically valid—and while the positivity of the prior density may, under mild regularity conditions, be viewed as equivalent to the KL support condition of \cite{schwartz1964consistency}—these assumptions are quite specific and restrictive (see, e.g., Section~\ref{sec:counterexample} for a demonstration of this point). More importantly, their formulation is not informed by any clear connection to the core mechanisms that govern posterior consistency, but rather aim to obtain the latter as a consequence of well-behaved frequentist procedures. Before presenting our own alternative treatment, we briefly review a more recent approach to consistency which, although potentially more broadly applicable, shares many of the same underlying limitations when applied to parametric models.

\subsection{An alternative MLE-based strategy}

Alternatively, \cite{walker2001} approached the problem of posterior consistency as follows. Taking $A_\varepsilon$ to be a Hellinger $\varepsilon$-ball around $f_{\theta_\star}$, consider
\begin{equation*}
    \Pi(A_\varepsilon^c \mid X_{1:n}) = \frac{\int_{A_\varepsilon^c}\prod_{i=1}^n\frac{f_\theta(X_i)}{f_{\theta_\star}(X_i)} \,\Pi(\mathrm d\theta)}{\int_{\Theta}\prod_{i=1}^n\frac{f_\theta(X_i)}{f_{\theta_\star}(X_i)} \,\Pi(\mathrm d\theta)}
\end{equation*}
for sufficiently small $\varepsilon > 0$. Standard results \citep[see, e.g.,][Lemma 4]{barron1999} provide suitable lower bounds for the denominator as long as $\theta_\star\in\kls(\Pi)$. For the numerator, letting $\hat\theta_n$ denote an MLE, one can write
\begin{equation*}
    \int_{A_\varepsilon^c}\prod_{i=1}^n\frac{f_\theta(X_i)}{f_{\theta_\star}(X_i)} \,\Pi(\mathrm d\theta) \leq \prod_{i=1}^n\left(\frac{f_{\hat\theta_n}(X_i)}{f_{\theta_\star}(X_i)}\right)^{1/2} \int_{A_\varepsilon^c}\prod_{i=1}^n\left(\frac{f_\theta(X_i)}{f_{\theta_\star}(X_i)}\right)^{1/2} \Pi(\mathrm d\theta).
\end{equation*}
While the second factor is readily shown to decay exponentially in $n$, the key idea in \cite{walker2001} is to require that
\begin{equation}\label{eq:walker_hjort_MLE_condition}
    \prod_{i=1}^n\left(\frac{f_{\hat\theta_n}(X_i)}{f_{\theta_\star}(X_i)}\right)^{1/2} < e^{cn/2} \quad \iff\quad \frac{1}{n}\sum_{i=1}^n \ln\left(\frac{f_{\hat\theta_n}(X_i)}{f_{\theta_\star}(X_i)}\right) < c
\end{equation}
eventually a.s.-$F_{\theta_\star}^\infty$ for all $c > 0$, which arises if
\begin{equation*}
\lim_{n\to\infty} \frac{1}{n}\sum_{i=1}^n  \ln\left(\frac{f_{\hat\theta_n}(X_i)}{f_{\theta_\star}(X_i)}\right) =  0
\end{equation*}
a.s.-$F_{\theta_\star}^\infty$. Under this condition, one obtains $\lim_{n \to \infty} \Pi(A_\varepsilon^c \mid X_{1:n}) = 0$ a.s.-$F_{\theta_\star}^\infty$ for all $\varepsilon > 0$, thus establishing posterior consistency.

The condition in \eqref{eq:walker_hjort_MLE_condition} is often satisfied by regular parametric models, for instance as guaranteed by the classical theorem of \cite{wilks1938large}. Nonetheless, this proof technique inherits the same core limitations as that of \cite{walker1969}: it relies on potentially stringent regularity conditions to guarantee MLE behavior \citep{kiefer1956consistency, vandegeer2000empirical, vandervaart2000asymptotic}, and achieves consistency through a sequence of technical bounds, rather than by appealing to a conceptual framework that directly explains why the posterior should concentrate. In the next section, we move beyond these approaches and develop a more principled and flexible route to consistency, rooted in the concept of sequential identifiability.

\section{Sequential identifiability, oscillations and posterior consistency}\label{sec:seq_identifiability}


As mentioned in Section~\ref{sec:introduction}, our starting point is the groundbreaking result of \cite{schwartz1964consistency}. Given its central role, we formally state it here (without proof).

\begin{theorem}[Schwartz]\label{thm:schwartz}
    Let $\theta_\star\in\kls(\Pi)$. Then the posterior is weakly consistent at $\theta_\star$, that is,
    \begin{equation*}
        \lim_{n \to \infty} \Pi(\{\theta \in \Theta : d_w(F_\theta, F_{\theta_\star}) > \varepsilon\} \mid X_{1:n}) = 0
    \end{equation*}
    a.s.-$F_{\theta_\star}^\infty$ for all \(\varepsilon > 0\).
\end{theorem}

While remarkable, Theorem~\ref{thm:schwartz} is not sufficient to address consistency in density estimation, as the topology induced by \(d_w\) is too weak to meaningfully capture closeness between densities. In particular, it is possible to construct sequences of densities that do not converge in the Hellinger sense, yet whose associated CDFs converge pointwise to that of a continuous random variable—i.e., they converge weakly. This phenomenon, which will be analyzed in greater detail in Theorem~\ref{thm:oscillations} and illustrated through a concrete model in Section~\ref{sec:counterexample}, arises from the fact that a sequence of densities may oscillate indefinitely in a manner that precludes convergence in Hellinger distance (hence in terms of parameter values), while still approximating the target CDF in distribution (see Figure~\ref{fig:kde_oscillations} in Subsection~\ref{sub:new_perspective}). In such cases, although the limit in density does not exist, the sequence may nonetheless approximate the CDF of the density around which it oscillates.

To address this potential gap between the two notions of convergence, a natural strategy is to rule out the possibility of such pathological approximations occurring within the parametric family of interest. The next definition, fundamental to our analysis, formalizes this idea.

\begin{definition}\label{def:sequential_identifiability}
    The parametric family $\mathcal F_\Theta$ is sequentially identifiable at $\theta_\star\in\Theta$ if, for any sequence $(\theta_j)_{j\in\NN}\subseteq \Theta$, $F_{\theta_j}\wc F_{\theta_\star}$ as $j\to\infty$ implies $\lim_{j\to\infty}\Vert\theta_j-\theta_\star\Vert = 0$.
\end{definition}

Intuitively, the concept of sequential identifiability implies that there exists no region of the Euclidean parameter space where $F_{\theta_\star}$ can be sequentially approximated in distribution, except in arbitrarily small neighborhoods of $\theta_\star$ itself. For example, if $\Theta = [0,\infty)$, sequential identifiability precludes the possibility that $F_\theta$ approximates $F_{\theta_\star}$ in distribution as $\theta \to \infty$. Moreover, observe that if $(\theta_j)_{j\in\NN}$ is a sequence such that $\lim_{j\to\infty} \theta_j = \theta \neq \theta'$, then our assumption that Euclidean convergence implies Hellinger convergence (and therefore weak convergence) ensures that, under sequential identifiability on all of $\Theta$, we must have $F_{\theta} \neq F_{\theta'}$. Thus, sequential identifiability may be interpreted as a strengthening of the standard notion of identifiability (Definition~\ref{def:identifiability}), requiring that no distribution in the model appears at multiple locations of the parameter space, once the latter is suitably extended to include weak limits.

It is worth noting that this seemingly “obvious” extension of the notion of identifiability has been largely overlooked for a simple reason: it offers no benefit for analyzing the convergence properties of the MLE. Because, as we have demonstrated, Bayesian consistency has traditionally been studied in close connection with MLE consistency, our notion in Definition~\ref{def:sequential_identifiability} would not have surfaced within that framework. However, as we show next, if sequential identifiability holds, it leads directly to Bayesian consistency. This observation highlights a key message of our work: Bayesian consistency and MLE consistency should be treated using fundamentally different tools, contrary to the prevailing approach in the literature. This point will be further illustrated in Section~\ref{sec:counterexample} with a concrete parametric model.

The next theorem is the central result of our analysis, establishing a key connection between sequential identifiability and posterior consistency.

\begin{theorem}\label{thm:general_parametric_consistency}
    If $\theta_\star\in\textnormal{KLS}(\Pi)$ and $\mathcal F_\Theta$ is sequentially identifiable at $\theta_\star$, the posterior is consistent at $\theta_\star$. In particular, the posterior predictive density
    $$\hat f_n:= \int_\Theta f_\theta\, \Pi(\mathrm d\theta\mid X_{1:n})$$
    is a consistent estimator of $f_{\theta_\star}$.
\end{theorem}

Theorem \ref{thm:general_parametric_consistency} is noteworthy in that it establishes posterior consistency under quite mild conditions: sequential identifiability of the model and a prior that assigns positive mass to any KL neighborhood of the true parameter, the latter being a standard well-specification assumption. As we illustrate in Section~\ref{sec:counterexample}, this condition captures the core mechanism behind Bayesian consistency and, unlike the classical assumptions in the literature, is entirely decoupled from MLE convergence. In fact, our illustrative model will reveal that even when sequential identifiability holds at a given parameter value, ensuring posterior consistency at it, the MLE may still be inconsistent due to likelihood peaks at the data. The intuition behind the derivation of Theorem~\ref{thm:general_parametric_consistency} is as follows: the KL support condition, via Schwartz’s theorem, ensures that the posterior concentrates in weak neighborhoods of $F_{\theta_\star}$. Sequential identifiability then rules out the possibility that such concentration occurs around points in the parameter space other than that corresponding to the true density, thereby yielding posterior consistency.

A practical implication of this result is that, to establish parameter consistency at some $\theta_\star$, it suffices to verify that the model space—augmented by its weak limit points—does not contain $F_{\theta_\star}$ at more than one location (e.g., both within the parameter space and along a sequence $(\theta_j)_{j \in \mathbb{N}}$ with $\lim_{j\to\mathbb N}\|\theta_j\| = \infty$). As we show in the next subsection, the possibility of such a scenario is tied to the presence of pathological oscillations of the model densities around $f_{\theta_\star}$, which is highly unlikely in any parametric modeling setting unless the modeler possesses specific knowledge of $f_{\theta_\star}$ and deliberately constructs the model to fail sequential identifiability at $\theta_\star$. Notice also that, in line with our earlier discussion in Subsection~\ref{sub:new_perspective}, a brief inspection of the proof of Theorem~\ref{thm:general_parametric_consistency} reveals that its validity is not restricted to the case of a finite-dimensional Euclidean parameter space $\Theta$, and in fact formally extends to nonparametric models as well. Nevertheless, as we have argued, sequential identifiability is generally not a tenable assumption in nonparametric settings, where, unlike in the parametric case, sequential \emph{un}identifiability is typically inherent.

Another insight from Theorem~\ref{thm:general_parametric_consistency} arises from the following observations. If the posterior fails to be consistent at some $\theta_\star \in \Theta$, Theorem~\ref{thm:general_parametric_consistency} implies the existence of a region in the augmented parameter space that induces a lack of sequential identifiability at $\theta_\star$, with the posterior accumulating with positive probability around such a weak limit point, separated from $\theta_\star$ in the Euclidean sense. Under the basic assumption of (traditional) identifiability, however, inconsistency cannot result in posterior concentration around a different point within the parameter space, as this, together with Schwartz's assurance of weak consistency, would contradict identifiability. Rather, the posterior mass must shift toward a region in the augmented space, lying outside the original parameter space, where the true distribution is also recovered. In other words, the posterior does not simply ``miss'' the true parameter by concentrating around a nearby but incorrect value; instead, it shifts toward regions in the augmented space proper, such as points at infinity, which correspond to weak limits of the true distribution. This observation suggests a practical heuristic for diagnosing strong consistency: in addition to directly verifying sequential identifiability, one may examine whether the posterior remains confined within reasonable regions of the parameter space, rather than drifting toward, for instance, infinity. In light of our discussion, such stability would provide evidence in support of consistency.

\subsection{The role of oscillations}

While sequential identifiability is the central concept of our theoretical analysis, it is crucial to understand the implications of its failure. The next theorem addresses this question by examining the behavior of a sequence of density functions that do not converge in the Hellinger metric (which, in our parametric setting, is equivalent to non-convergence of the associated parameter sequence), yet whose corresponding distribution functions do converge to a proper limit.

To state the next result, we first establish the following pieces of terminology. For any two densities $f$ and $g$ such that $A:=\{x\in\RR : g(x)>f(x)\}$ is open, we say that $g$ oscillates $O$ times around $f$ if $O\in\NN$ is the minimum number of disjoint intervals $(a_1, b_1), \dots, (a_O, b_O)$ such that we can write
\begin{equation*}
    A = \bigcup_{i=1}^O (a_{i}, b_{i}).
\end{equation*}
Notice that the openness of $A$ implies that there exists a decomposition of it into countably many disjoint open intervals, and we call the above expression the \textit{minimal decomposition} of $A$.

\begin{theorem}\label{thm:oscillations}
    Let $g, f_1, f_2,\dots$ be densities such that the sets
    \begin{equation*}
        A_j:=\{x\in\RR : f_j(x)>g(x)\}, \quad B_j:=\{x\in\RR : g(x)>f_j(x)\}
    \end{equation*}
    are open for all $j\in\NN$.\footnote{While openness of $A_j$ and $B_j$ is a minimal requirement needed in the proof of the result, notice that any set of densities $g, f_1, f_2, ...$ that are continuous on a common support satisfies the assumption.} Moreover, assume that (i) $d_w(F_j, G) \to 0$ as $j\to\infty$, and (ii) $d_h(f_j, g) \geq \varepsilon >0$ for all $j\in\NN$. Then the number of oscillations $O_j$ of $f_j$ around $g$ tends to infinity as $j\to\infty$.
\end{theorem}

The preceding result implies that, if sequential identifiability fails at some \(\theta \in \Theta\), then the model must contain a sequence of densities that oscillate arbitrarily frequently around \(f_\theta\). Figure~\ref{fig:kde_oscillations} in Subsection~\ref{sub:new_perspective} provides a graphical illustration of the behavior of such a sequence. Consequently, ruling out this kind of pathological behavior is sufficient to guarantee posterior consistency. In practice, due to the inherently limited expressive power of finite-dimensional models (with respect to weakly approximating distributions outside the proper parameter space), posterior inconsistency from this mechanism could arise only if the modeler had precise knowledge of the true density and intentionally introduced carefully constructed oscillations around it. Notice that these oscillations would not only need to be present, but also to integrate in such a way as to yield the correct distribution function away from the true parameter value. Clearly, such a contrived construction is implausible in any real-world parametric modeling scenario, rendering our simple sequential identifiability condition effectively universal for parametric consistency.

The role of oscillations will be further explored in Section~\ref{sec:counterexample} using a simple illustrative parametric model. Before doing so, we review a few common parametric families and show how consistency can be easily established using the conditions introduced in this section.

\subsection{Example 1: exponential families}

Assume that, for each $\theta \in \Theta\subseteq \RR^p$, the associated density satisfies
\begin{equation*}
    f_\theta(x) \propto h(x) \exp\left\{\theta^\top T(x)\right\}, \quad x\in\RR.
\end{equation*}
Then $\mathcal{F}_\Theta$ defines a $d$-dimensional exponential family with sufficient statistics $T(x) \in \mathbb{R}^p$ \citep{efron2022}. This general form includes several classical parametric models for continuous data—such as the Gaussian, exponential, gamma, beta, Laplace, Rayleigh, Weibull, and von Mises distributions—for which sequential identifiability (and therefore posterior consistency) is readily verified.

\subsection{Example 2: uniform distribution}

Another basic parametric model not expressible as an exponential family is the uniform distribution on $[0,\theta]$ for $\theta \in \Theta = (0, \infty)$. Specifically, assume $f_\theta(x) \propto 1_{[0,\theta]}(x)$. We show that sequential identifiability holds by contraposition: a sequence $(\theta_j)_{j\in\NN}$ does not converge in $\Theta$ either if $\liminf_{j\to\infty}\theta_j\neq\limsup_{j\to\infty}\theta_j$ (in which case $(F_{\theta_j})_{j\in\NN}$ clearly has no weak limit), or if $\lim_{j\to\infty}\theta_j = \ell \in\{0,\infty\}$. If $\ell=0$, then $f_{\theta_j}\wc \delta_0$, which does not admit density and therefore does not belong to the model. If instead $\ell=\infty$, $(F_{\theta_j})_{j\in\NN}$ is not tight and, by Prokhorov's theorem, it does not converge weakly to any probability distribution.

\subsection{Example 3: finite mixture models}

Consider a normal mixture model with a finite number $K$ of components. For ease of exposition, we circumvent the usual parameter identifiability issues associated with mixtures \citep{teicher1963identifiability} and work directly on the space of mixture densities equipped with the Hellinger metric:\footnote{Nevertheless, by imposing standard identifiability constraints on the mixture parameters, the following analysis can be extended to the usual Euclidean setting. We also note that the same line of reasoning applies to mixtures with more general kernel families.} specifically, the $K$-component normal mixture model is defined as
\begin{equation*}
    \mathcal F^K := \left\{\sum_{k=1}^K w_k\, \lambda_k^{1/2}\phi\left(\lambda_k^{1/2}(\cdot-\mu_k)\right) : (\boldsymbol{w}, \boldsymbol{\mu}, \boldsymbol{\lambda}) \in\Delta_K\times\RR^K\times(0,\infty)^K \right\},
\end{equation*}
where $\Delta_K$ denotes the $(K-1)$-dimensional simplex and $\phi(x) := (2\pi)^{-1/2}e^{-x^2/2}$ for all $x \in \RR$. Since any two densities $f, g \in \mathcal F^K$ are continuous, we may invoke Theorem~\ref{thm:oscillations} to establish sequential identifiability, and hence posterior consistency. Indeed, for any fixed $K \in \NN$, the number of oscillations of $f$ around $g$ is bounded above by a finite constant (common to all $f,g\in\mathcal F^K$), implying that weak convergence within $\mathcal F_\Theta$ entails Hellinger convergence. Consequently, by Theorem~\ref{thm:general_parametric_consistency}, any density in the KL support of the prior will exhibit posterior consistency.

\section{An illustrative model}\label{sec:counterexample}

We are now in a position to illustrate our theory with a simple yet instructive parametric model. Specifically, we restrict the sample space to $[0,1]$ and consider the one-dimensional parameter space $\Theta = [0, \infty)$, where the family $\mathcal F_\Theta$ consists of densities of the form
\begin{equation}\label{eq:sin_density}
    f_\theta(x) = \frac{1 + \cos(\theta x)}{1 + \sin(\theta)/\theta}.
\end{equation}
To be precise, we adopt the convention that $f_0$ is defined as the continuous extension of the above expression as $\theta \to 0$, yielding $f_0(x) = 1$ for all $x \in [0,1]$—that is, $F_0$ corresponds to the uniform distribution. The resulting family of CDFs is of the form
\begin{equation*}
    F_\theta(x) = \frac{x + \sin(\theta x)/\theta}{1 + \sin(\theta)/\theta}, \quad \forall x\in[0,1].
\end{equation*}

As illustrated in Figure~\ref{fig:cos_densities}, increasing values of $\theta$ induce more pronounced oscillations in the density around the value 1, while the corresponding CDF converges to $F_0$ (cf.\ Figure \ref{fig:kde_oscillations} in Section~\ref{sec:seq_identifiability}). We summarize these and other basic properties of the model in the next proposition.

\begin{figure}
    \centering
    \includegraphics[width=0.75\linewidth]{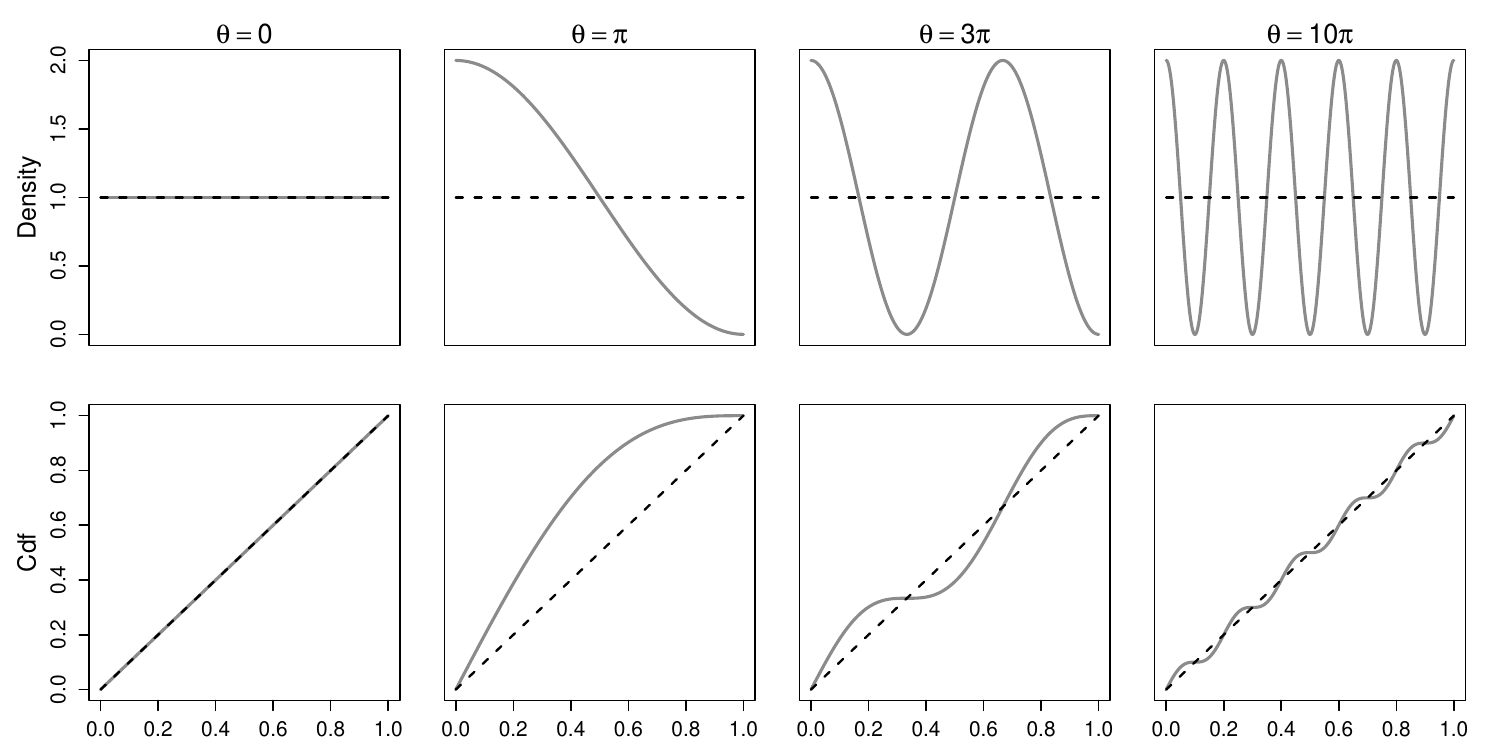}
    \caption{Density functions (solid gray, top row) and CDFs (solid gray, bottom row) in the illustrative parametric model, shown for increasing values of the parameter $\theta$. As $\theta$ grows, the density becomes increasingly oscillatory while the CDF converges to $F_0$ (dashed black).}
    \label{fig:cos_densities}
\end{figure}

\begin{proposition}\label{pro:properties_counterexample}
    The parametric family $\mathcal F_\Theta$ defined by \eqref{eq:sin_density} satisfies the following properties:
    \begin{enumerate}
        \item The Euclidean and Hellinger metrics are equivalent on $\mathcal F_\Theta$;
        \item For any $\theta \geq 0$ and prior $\Pi$ on $\Theta$, if $\Pi(A) > 0$ for every Euclidean neighborhood $A$ of $\theta$, then $\theta\in\kls(\Pi)$;
        \item $\mathcal F_\Theta$ is sequentially identifiable at all $\theta > 0$;
        \item As $\theta \to \infty$, $F_\theta \wc F_0$, although $f_\theta$ does not converge to any density in the Hellinger sense.
    \end{enumerate}
\end{proposition}

The first property confirms that the assumption of equivalence between the Euclidean and Hellinger metrics, made throughout this article, holds in this setting as well. This allows us to move freely between the two notions of convergence without loss of generality. The second property implies that, for the KL support condition to hold at all parameter values, it suffices for the prior to have full support on $\Theta$. Combined with the third property, this enables us to invoke Theorem~\ref{thm:general_parametric_consistency} and obtain the following corollary.

\begin{corollary}
    If the prior $\Pi$ has full support on $[0,\infty)$, then the posterior is consistent at all $\theta > 0$.
\end{corollary}

This result is remarkable because, as we demonstrate in Subsection~\ref{sub:failure_conditions} below, the family $\mathcal F_\Theta$ fails to satisfy the classical sufficient conditions for consistency proposed in earlier literature (cf.\ Section~\ref{sec:classical_conditions}). Nonetheless, our simple criterion of sequential identifiability seamlessly yields posterior consistency at all $\theta>0$.

Due to the fourth property in Proposition~\ref{pro:properties_counterexample}, however, the same argument does not extend to \(\theta = 0\), as the model is sequentially unidentifiable at that point. In particular, while a prior with full support guarantees posterior concentration in weak neighborhoods of \(F_0\) (thanks to Schwartz's theorem), such concentration may occur in the ``wrong'' region of the parameter space—namely, toward infinity. In line with Theorem~\ref{thm:oscillations}, this happens because the cosine-based model oscillates around the uniform density with increasing frequency as the parameter \(\theta\) diverges.

\subsection{Some remarks on the illustrative model}

Before continuing the formal analysis of our example, a few important remarks about its construction are in order. First, the density model defined by \eqref{eq:sin_density} demonstrates that, although sequential unidentifiability is theoretically possible, it is extremely unlikely to occur in any reasonably specified parametric model, and thus poses little practical threat to posterior consistency. In our example, inconsistency can arise only if the true parameter is $\theta_\star = 0$, a scenario we engineered by precisely tailoring cosine-based oscillations around the corresponding (uniform) density. Such a construction is implausible in real-world settings: not only are densities with pathologically frequent oscillations rarely found in practical parametric models, but even if present, it is highly unlikely that they would align so precisely with the true (and unknown) density. Since a parametric model has only limited capacity to generate these oscillations around any given target, the conditions for inconsistency require both unlikely model structure and foreknowledge of the data-generating process---making sequential unidentifiability an essentially self-inflicted phenomenon for the purposes of parametric inference.

A more subtle point arises by noticing that our construction can be generalized by creating problematic oscillatory behavior around an entire parametric family, rather than just the uniform distribution used in our example. For instance,\footnote{The following construction starts from a simple parametric family $\mathcal G$ to keep a clear notation, though it can be further extended to more general families.} consider the parametric family $\mathcal{G}$ containing all densities of the form
\begin{equation*}
    g_\lambda(x) = \lambda^{-1} 1_{[0, \lambda]}(x)
\end{equation*}
for $\lambda \in \Lambda \subseteq [1, 2]$, and construct a new model $\mathcal{H}$ containing all densities of the form
\begin{equation*}
    h_{\lambda, \omega}(x) = \frac{g_\lambda(x) + \mu \cos(\omega x)}{1 + \mu \sin(\omega \lambda)/\omega} \, 1_{[0, \lambda]}(x)
\end{equation*}
for $\lambda \in \Lambda$ and $\omega \in \Omega := [0, \infty)$, where $\mu \in (0,1)$ is a constant small enough (depending on $\Lambda$) to ensure positivity of all members of $\mathcal{H}$. Our illustrative example is then recovered if $\Lambda = \{1\}$ and $\mu=1$. The associated CDFs satisfy
\begin{equation*}
    H_{\lambda, \omega}(x) = \frac{x/\lambda + \mu \sin(\omega x)/\omega}{1 + \mu \sin(\omega \lambda)/\omega}
\end{equation*}
for $x \in [0, \lambda]$, so that $H_{\lambda, 0} = G_\lambda$ (by continuous extension) and $H_{\lambda, \omega} \wc G_\lambda$ (while densities do not converge) as $\omega \to \infty$, for all $\lambda \in \Lambda$. That is, by introducing an auxiliary oscillation parameter $\omega$, one is able to induce sequential unidentifiability at every member of the original family $\mathcal{G}$. While this procedure makes it possible to create sequential unidentifiability on a continuum of densities rather than just a few isolated ones,\footnote{Notice that sequential unidentifiability still only holds on a proper \emph{sub}space of the parameter space of $\mathcal{H}$---that is, on a lower-dimensional subset---as opposed to holding everywhere, as is typical in nonparametric models.} the construction still reinforces our key message: starting from any parametric family of actual modeling interest, such as $\mathcal{G}$ when estimating the support of a uniform distribution, inducing sequential unidentifiability requires an adversarial effort involving deliberate expansion of the parameter space to engineer oscillations. Thus, while sequential unidentifiability is theoretically possible, it remains an artificial phenomenon in any realistic parametric modeling context.

\subsection{Failure of classical conditions}\label{sub:failure_conditions}

While our principled analysis immediately established consistency for all \(\theta > 0\), we now show that the oscillatory behavior of the densities defined in \eqref{eq:sin_density} violates the classical regularity conditions commonly used to ensure posterior consistency. As a consequence, these existing results are inadequate for analyzing posterior convergence in this model.

In particular, with respect to the conditions introduced by \cite{walker1969}, we have the following proposition.

\begin{proposition}\label{pro:failure_walker}
    The parametric model in \eqref{eq:sin_density} fails assumptions~\Wref{walker_ass1} and~\Wref{walker_ass3} (cf.\ Section~\ref{sec:classical_conditions}).
\end{proposition}

Heuristically, the failure of assumption~\Wref{walker_ass1} stems from the model’s oscillations causing the density to vanish at varying points on the support. The failure of assumption~\Wref{walker_ass3}, on the other hand, arises from the persistence of oscillations in the likelihood function even as \(\theta \to \infty\).

More fundamentally, we next establish that, under this model, the MLE is inconsistent at all data-generating parameter values. This result reinforces the limitations of the approach developed by \cite{walker1969}, which relies on model regularity to infer posterior consistency from MLE consistency. The inconsistency of the MLE renders the methodology of \cite{walker2001} inapplicable as well, since that framework also presupposes regular MLE behavior.

\begin{theorem}\label{thm:MLE_inconsistency}
    For all true data-generating parameters \(\theta_\star \geq 0\), maximum likelihood estimation is inconsistent. Specifically, for all \(M > 0\),
    \begin{equation*}
        \max_{\theta \in [0, M]} \prod_{i=1}^n f_\theta(X_i) < \sup_{\theta \geq 0} \prod_{i=1}^n f_\theta(X_i)
    \end{equation*}
    ultimately a.s.-\(F_{\theta_\star}^\infty\).
\end{theorem}

The proof of Theorem~\ref{thm:MLE_inconsistency} relies on a classical number-theoretic result, namely Dirichlet's simultaneous approximation theorem \citep{schmidt2009diophantine}. Heuristically, the argument proceeds in two steps. First, we show that any MLE based on \(n\) observations must achieve a product likelihood of at least \(2^n\), as it is always possible to select a value of \(\theta\) such that each data point aligns with a peak of the cosine oscillations (which, as can be readily verified, reach values of around 2). Second, we demonstrate that, for all sufficiently large \(n\), with probability 1, no likelihood maximizer restricted to \([0, M]\) for any fixed \(M > 0\) can attain this lower bound. This establishes the asymptotic failure of the maximum likelihood principle—and, consequently, of classical approaches to posterior consistency—at all $\theta_\star \geq 0$.

\subsection{Consistency at the uniform density}\label{sub:consistency_at_0}

As we have already discussed, while all classical approaches fail to ensure posterior consistency for the model under consideration, our Theorem~\ref{thm:general_parametric_consistency} easily establishes the result for all $\theta_\star > 0$. At $\theta_\star = 0$, however, we have shown that sequential identifiability is violated, as $F_\theta \wc F_0$ when $\theta \to \infty$, and therefore Theorem~\ref{thm:general_parametric_consistency} does not directly apply. Nonetheless, we now demonstrate that, by employing standard techniques from the literature on nonparametric models, consistency at $\theta_\star$ can still be obtained under very mild conditions on the prior.

The first approach we consider is the one introduced in the seminal papers by \cite{ghosal1999} and \cite{barron1999}, which establishes exponential decay of the posterior mass outside Hellinger neighborhoods of $\theta_\star$ by splitting the complement of such neighborhoods into two $n$-dependent regions: the first has slowly growing complexity—e.g., measured in terms of Hellinger metric or upper-bracketing entropy \citep{vandegeer2000empirical, wainwright2019high}—forcing the likelihood ratio to decay exponentially; the second  region is directly assumed to have exponentially small prior (hence posterior) mass. Together with the KL support assumption and an upper-bound on $d_h(f_\theta, f_{\theta'})$ by the corresponding Euclidean distance (see the supplementary material for more details), this leads to the following result.

\begin{proposition}\label{pro:G-VdV_consistency}
    Assume that $0 \in \kls(\Pi)$ and that, for some function\footnote{Valid choices of $\varphi$ include $\varphi(t) = t^{1+\beta}$ and $\varphi(t) = \beta \exp(t)$ for some $\beta > 0$. The latter condition, for example, is satisfied by an exponential prior.} $\varphi : (0, \infty) \to (0,\infty)$ such that $\lim_{t \to \infty} \varphi(t)/t = \infty$, the prior CDF satisfies
    \begin{equation*}
        1 - \Pi(\theta) \leq e^{-\varphi(\ln \theta)}
    \end{equation*}
    for all sufficiently large $\theta > 0$. Then the posterior is consistent at $\theta_\star = 0$.
\end{proposition}

A second approach is that of \cite{walker2004squarerootsum}, who showed that posterior consistency holds as long as the parameter space can be covered by sets whose prior masses satisfy a suitable summability condition. In our setting, this yields the following result.

\begin{proposition}\label{pro:consistency_walker2004}
    Assume that $0 \in \kls(\Pi)$ and that the prior density $\pi$, for all sufficiently large $\theta > 0$, is decreasing and satisfies
    \begin{equation}\label{eq:tail_condition_sqrt}
        \pi(\theta) \lesssim \frac{1}{\theta^2 (\ln \theta)^{2 + \beta}}
    \end{equation}
    for some $\beta > 0$. Then the posterior is consistent at $\theta_\star = 0$.
\end{proposition}

Finally, we consider a third strategy that leverages insights from our analysis of oscillations in the product likelihood, as well as structural properties of the model, to establish consistency under even weaker assumptions on the prior. Specifically, we exploit the observation that the product likelihood, when restricted to values of $\theta$ smaller than $e^{cn}$ (for some $c > 0$), is asymptotically well behaved: in this region, oscillations eventually vanish, and the restricted MLE exhibits regular behavior. Instead, for the region where $\theta > e^{cn}$, any prior with sub-polynomial tails suffices to ensure the necessary posterior mass decay, leading to the following result.

\begin{theorem}\label{thm:our_consistency}
    Assume that $0 \in \kls(\Pi)$ and that the prior density \(\pi\), for all sufficiently large \(\theta > 0\), satisfies  
    \[
    \pi(\theta) \lesssim \theta^{-(1+\alpha)}
    \]  
    for some $\alpha > 0$. Then the posterior is consistent at $\theta_\star = 0$.
\end{theorem}

\section{Discussion}\label{sec:discussion}

We introduced a general framework for studying posterior consistency in parametric models, centered on the simple yet powerful notion of sequential identifiability. Departing from the classical approach that links posterior consistency to the regular behavior of maximum likelihood estimation, our perspective builds on the foundational result of \citet{schwartz1964consistency}, which guarantees weak consistency under a mild Kullback-Leibler support condition. Sequential identifiability then arises as the minimal additional assumption needed to lift weak consistency to consistency at the level of parameters. We also showed that inconsistency can only occur when the model admits self-inflicted pathological oscillatory behavior around the true density. This insight, along with the identifiability criterion itself, allows us to establish consistency under assumptions significantly less restrictive than those required by classical regularity-based theories, which were primarily designed to prevent MLE overfitting as a consequence of likelihood peaks at the data points. Our framework thus opens the door to analyzing models that fall outside the scope of existing results, including the parametric example we constructed and examined in detail.

Our work has a multiplicity of ramifications. Most directly, it offers novel and accessible guidance for the applied Bayesian statistician seeking to understand the asymptotic behavior of parametric models of interest. Building on our analysis of sequential identifiability, our message to the modeler is straightforward: rather than verifying a list of regularity conditions, simply avoid introducing arbitrarily oscillating densities into the model in the first place. Even when such oscillations are unavoidable, parametric inconsistency remains unlikely unless the modeler is able to design oscillations that align with the unknown data-generating process. However, in that scenario, statistical inference under a parametric model becomes questionable. Either the modeler is confident in using a finite-dimensional family and yet somehow knows the true distribution well enough to adversarially target it with pathological oscillations, making inference hardly necessary; or the modeling goal demands such flexibility that sequential unidentifiability arises over a broad subset of the parameter space, in which case a nonparametric approach may be more suitable. Our illustrative example required exactly this kind of contrived construction to potentially induce inconsistency, and even then, only at a single, sequentially unidentifiable parameter value. Remarkably, we further showed that consistency can still be recovered at that value through a detailed analysis of the model’s oscillatory behavior and appropriate control via the prior. This was done by means of techniques inspired by the study of nonparametric consistency, where this kind of extreme oscillatory behavior naturally belongs.

Second, as we have discussed in detail, our analysis bears strong connections with the literature on nonparametric posterior consistency. Specifically, by recognizing some fundamental differences between parametric and nonparametric models, we have arrived at sequential identifiability as a key condition for parametric consistency, while recognizing its inappropriateness in nonparametric settings. Nevertheless, our analysis on oscillations, motivated by the need to understand the implications of the failure of sequential identifiability, is intimately tied with the most common approaches to nonparametric consistency. Although not explicitly, the sieve complexity conditions introduced by \cite{barron1999}, \cite{ghosal1999}, and \cite{walker2004squarerootsum} aim to address the same underlying issue: in those works, oscillatory behavior in the density model is quantified and controlled, for instance, via the Hellinger metric entropy of sets where most of the prior mass is concentrated. This is analogous to our calculations in Section~\ref{sec:counterexample}, where we established posterior consistency at the only parameter value for which our illustrative model is not sequential identifiable.

Finally, our analysis—particularly the construction of a model based on cosine oscillations—sheds new light on a well-known counterexample by \cite{barron1999}, in which the posterior is weakly consistent but not consistent in Hellinger distance. In that example, positive prior mass is placed on each member of a family of discontinuous, oscillatory densities on $[0,1]$ that alternate between values of 0 and 2 across disjoint intervals, thereby weakly approximating the uniform density. In our terminology, this corresponds to a model that is sequentially unidentifiable at the uniform density, and the mechanism driving inconsistency is closely related to our example, in which continuous oscillations occur between 0 and (approximately) 2 around the uniform density. Nevertheless, our model has been shown to achieve consistency at the uniform distribution under mild tail conditions on the prior, while the carefully placed prior mass on discontinuous, oscillatory densities in the example of \cite{barron1999} results in inconsistency. Although beyond the scope of this study, these observations suggest the possibility that additional mild assumptions on the nature of oscillations (e.g., continuity) or on the prior (e.g., no positive mass on single parameters or densities) may help to rule out inconsistent posterior behavior even in nonparametric settings.

To summarize our findings, for a parametric model to be inconsistent, not only must a distribution \( F_\infty \) exist at the weak boundary of the model---say, as the parameter diverges to infinity---but this distribution must also coincide with the one from which the sample is generated. As we have shown, this alignment can only result from an implausible and adversarial model design involving oscillations carefully tailored around the true, unknown density. Even in such cases, posterior inconsistency does not arise unless the prior assigns sufficient mass to these oscillatory features. For instance, in a finite Gaussian mixture model with \( K \) components, the boundary distributions \( F_\infty \) corresponds to discrete measures with at most \( K \) atoms, which can be ruled out as plausible true distributions because they do not admit a density. In our cosine-based example, where the oscillations are deliberately engineered so that \( F_0 = F_\infty \) correspond to the uniform distribution, we showed that inconsistency at $\theta=0$ is still ruled out even under priors with heavier-than-Cauchy tail, due to insufficient mass being placed on the problematic region. In short, parametric models, unlike nonparametric ones, may be regarded as universally consistent unless one adopts a highly artificial construction that simultaneously entangles the true distribution, the density model, and the prior.


\clearpage
\appendix

\begin{center}
{\bf{\Large{Supplementary Material for ``Posterior Consistency in Parametric Models via a Tighter Notion of Identifiability''}}}
\end{center}

In this supplementary material, we present the proofs of all the theoretical results from the main text of the article.

\subsection*{Proof of Theorem~\ref{thm:general_parametric_consistency}} %

By Theorem~\ref{thm:schwartz}, the KL support assumption implies that, for all $\varepsilon>0$,
\begin{equation*}
    \Pi(\{\theta\in\Theta : d_w(F_{\theta_\star}, F_\theta) \leq  \varepsilon\} \mid X_{1:n}) \to 1
\end{equation*}
a.s.-$F_{\theta_\star}^\infty$. Specifically, choosing any positive decreasing sequence $\varepsilon_j\to 0$,
\begin{equation*}
    F_{\theta_\star}^\infty\left(x_{1:\infty}\in\RR^\NN:\, \Pi(\{\theta\in\Theta : d_w(F_{\theta_\star}, F_\theta) \leq  \varepsilon_j\} \mid x_{1:n}) > 1-\delta \textnormal{ ultimately}\right) = 1.
\end{equation*}
for all $\delta>0$ and $j\in\NN$. Now fix $j\in\NN$ and assume per contra that the posterior is not consistent in the Euclidean metric, so that there exist $\varepsilon>0$ and $\delta>0$ such that
\begin{equation*}
    F_{\theta_\star}^\infty\left(x_{1:\infty} \in\RR^\NN:\,\Pi(\{\theta\in\Theta : \Vert\theta_\star- \theta\Vert >  \varepsilon\} \mid x_{1:n}) > \delta \textnormal{ i.o.}\right) >0.
\end{equation*}
Because $F_{\theta_\star}^\infty$ is a probability measure, the two above expressions imply that there exists some $x_{1:\infty}$ such that
\begin{align*}
    \Pi(\{\theta\in\Theta : d_w(F_{\theta_\star}, F_\theta) \leq  \varepsilon_j\} \mid x_{1:n}) & > 1-\delta \quad \textnormal{ultimately}, \\
    \Pi(\{\theta\in\Theta : \Vert\theta_\star- \theta\Vert >  \varepsilon\} \mid x_{1:n}) &> \delta \quad \textnormal{ i.o.}
\end{align*}
This implies the existence of some $n\in\NN$ such that
\begin{align*}
    \Pi(\{\theta\in\Theta : d_w(F_{\theta_\star}, F_\theta) \leq  \varepsilon_j\} \mid x_{1:n}) & > 1-\delta, \\
    \Pi(\{\theta\in\Theta : \Vert\theta_\star- \theta\Vert >  \varepsilon\} \mid x_{1:n}) & > \delta.
\end{align*}
Because $\Pi(\cdot\mid x_{1:n})$ is a probability measure, the last two inequalities imply the existence of some $\theta_j\in\Theta$ such that both $d_w(F_{\theta_\star}, F_{\theta_j}) \leq  \varepsilon_j$ and $\Vert\theta_\star- \theta_j\Vert >  \varepsilon$. Because we can repeat the above argument for all $j\in\NN$, we can construct a sequence $(\theta_j)_{j\in\NN}$ such that $d_w(F_{\theta_\star}, F_{\theta_j}) \leq  \varepsilon_j\to 0$ but $\inf_{j\in\NN}\Vert\theta_\star- \theta_j\Vert \geq  \varepsilon>0$. This leads to a contradiction of sequential identifiability at $\theta_\star$, concluding the proof of posterior consistency.

As for the consistency of $\hat f_n$, Jensen's inequality applied to the convex map $f \mapsto d_h(f, f_{\theta_\star})$, together with the definition $A_\varepsilon := \{\theta \in \Theta : d_h(f_\theta, f_{\theta_\star}) < \varepsilon\}$ for $\varepsilon > 0$, implies
\begin{align*}
    d_h(\hat f_n, f_{\theta_\star}) & \leq \int_\Theta d_h(f_\theta, f_{\theta_\star})\,\Pi(\mathrm d\theta \mid X_{1:n}) \\
    & = \int_{A_\varepsilon} d_h(f_\theta, f_{\theta_\star})\,\Pi(\mathrm d\theta \mid X_{1:n}) + \int_{A_\varepsilon^c} d_h(f_\theta, f_{\theta_\star})\,\Pi(\mathrm d\theta \mid X_{1:n}) \\
    & \leq \varepsilon + \sqrt{2}\, \Pi(A_\varepsilon^c \mid X_{1:n}).
\end{align*}
The second term converges to zero a.s-$F_{\theta_\star}^\infty$ for any $\varepsilon > 0$, by posterior consistency, while the first term can be made arbitrarily small. This completes the proof.

\subsection*{Proof of Theorem~\ref{thm:oscillations}} %

For convenience, recall the definition of the Lévy-Prokhorov distance between distributions $F_j$ and $G$:
\begin{equation*}
    d_w(F_j, G) := \inf\left\{\delta>0:\, \forall A\in\mathscr B(\RR),\, F_j(A)\leq G(A^\delta) +\delta, \, G(A)\leq F_j(A^\delta)+\delta\right\},
\end{equation*}
where $A^\delta := \{x\in\RR : \exists y\in A, \, \vert x-y\vert < \delta\}$. Therefore, calling $\delta_j:= d_w(F_j, G)$, by assumption we get\footnote{Notice that, if the infimum in the definition of $d_w$ is not attained by $\delta_j$, it suffices to replace $\delta_j$ with $\delta_j + \varepsilon_j$, for some non-negative sequence $\varepsilon_j \to 0$, in the following analysis. Therefore, without loss of generality, we continue to work with $\delta_j$.}
\begin{equation}\label{eq:LP_metric condition}
    F_j(A) \leq G(A^{\delta_j}) + \delta_j \textnormal{ and } G(A) \leq F_j(A^{\delta_j}) + \delta_j, \qquad \forall A\in\mathscr B(\RR).
\end{equation}
Now recall that, for probability measures on $\RR$ admitting a density, the Hellinger distance is topologically equivalent to the total variation distance $d_{TV}$, so assume without loss of generality that
\begin{equation*}
    \varepsilon \leq d_{TV}(F_j, G) := \sup_{A\in\mathscr B(\RR)} \vert F_j(A) - G(A)\vert \equiv \frac{1}{2}\int_\RR \vert f_j(x) - g(x) \vert \mathrm dx.
\end{equation*}
The above characterization of $d_{TV}$ implies that the $\sup$ in its definition is attained either by $A_j$ at $F_j(A_j) - G(A_j)\geq\varepsilon$, or by $B_j$ at $G(B_j) - F_j(B_j)\equiv F_j(A_j) - G(A_j) \geq\varepsilon$. Without loss of generality, assume that the $\sup$ is attained by $A_j$ and
\begin{equation}\label{eq:TV_condition}
    F_j(A_j) - G(A_j)\geq\varepsilon,
\end{equation}
the other case being perfectly symmetric. Denote by
\begin{equation*}
    A_j = \bigcup_{i=1}^{O_j} (a_{ij}, b_{ij}),
\end{equation*}
the minimal decomposition of $A_j$, which exists by the assumed openness of $A_j$. Then Equations~\eqref{eq:LP_metric condition} and \eqref{eq:TV_condition} combine into $   \varepsilon \leq G(A_j^{\delta_j}\setminus A_j) + \delta_j$, where
\begin{equation*}
    A_j^{\delta_j}\setminus A_j \subseteq \bigcup_{i=1}^{O_j}\Big([a_{ij} - \delta_j, a_{ij}]\cup [b_{ij}, b_{ij} + \delta_j]\Big)
\end{equation*}
is such that
\begin{equation*}
    G(A_j^{\delta_j}\setminus A_j) \leq 2 O_j\sup_{x\in\RR}\left\{G(x+\delta_j) - G(x)\right\}.
\end{equation*}
Therefore
\begin{equation*}
    \varepsilon\leq 2O_j\sup_{x\in\RR}\left\{G(x+\delta_j) - G(x)\right\} + \delta_j
\end{equation*}
for all $j\in\NN$. Because the distribution $G$ is a continuous, bounded and monotonically increasing function, it is also uniformly continuous, so $\lim_{j\to\infty} \delta_j = 0$ implies $\lim_{j\to\infty}\sup_{x\in\RR}\left\{G(x+\delta_j) - G(x)\right\} = 0$. This, together with the last expression, yields $\lim_{j\to\infty} O_j = \infty$.

\subsection*{Proof of Proposition~\ref{pro:properties_counterexample}}%

We prove that each property holds separately.

\subsubsection*{Proof of property 1}

To prove the equivalence of the Euclidean and Hellinger metrics, we rely on the following lemma.

\begin{lemma}\label{lem:hellinger<euclidean}
    For all $\theta, \theta' \geq 0$, $d_h(f_\theta, f_{\theta'}) \leq \min\left\{\sqrt 2,\, |\theta - \theta'|\right\}$.
\end{lemma}

\begin{proof}
    $d_h(f_\theta, f_{\theta'}) \leq \sqrt 2$ holds by design, while checking the other upper-bound is straightforward once one verifies that the family of functions $\{\theta \mapsto \sqrt{f_\theta(x)} : x \in [0,1]\}$ is uniformly 1-Lipschitz continuous. To get this, it suffices to check that
    \begin{equation*}
        \left|\frac{\partial}{\partial \theta}\sqrt{f_\theta(x)}\right|=\left|\frac{x\sin(\theta x)}{2\sqrt{1 + \cos(\theta x)}\sqrt{1 + \sin (\theta)/\theta}} + \frac{\sqrt{1 + \cos(\theta x)}(\theta\cos(\theta) -\sin (\theta))}{2\theta^2(1 + \sin(\theta)/\theta)^{3/2}}\right|,
    \end{equation*}
    which exists almost everywhere for all $x \in [0,1]$, is bounded above by 1 for all $x \in [0,1]$ and all $\theta \geq 0$ at which it exists, and moreover that $\theta\mapsto \sqrt{f_\theta(x)}$ is absolutely continuous on $[0,\infty)$ for all $x\in[0,1]$. Plugging this into the definition of $d_h$ yields the desired result.
\end{proof}

In particular, the preceding lemma shows that $d_h(f_\theta, f_{\theta'}) \to 0$ if $|\theta - \theta'| \to 0$. The reverse implication follows from the fact that, for $\theta>0$, Hellinger convergence to $f_\theta$ implies weak convergence to $F_\theta(x) = (x + \sin(\theta x)/\theta)/(1 + \sin(\theta)/\theta)$, which can only happen if the parameter value converges to $\theta$ itself. As for $\theta = 0$, instead, property 4 shows that weak convergence to $F_0$ can only happen as parameter values converge to 0 or as they diverge to $\infty$. However, in the latter case, Hellinger convergence fails (see again property 4), proving that $d_h(f_\theta, f_0)\to 0$ implies $\theta\to 0$.

\subsubsection*{Proof of property 2}

Our aim is to show that $\lim_{\theta'\to \theta}\kl(f_\theta, f_{\theta'}) =0$ for all $\theta$. If that is the case, for all $\varepsilon>0$ small enough there exists $\delta>0$ so that $|\theta'-\theta|<\delta$ implies $\kl(f_\theta, f_{\theta'})<\varepsilon$, or
    $$\{\theta'\geq 0 : |\theta'-\theta|<\delta\}\subseteq \{\theta'\geq 0 : \kl(f_\theta, f_{\theta'})<\varepsilon\}.$$
Because the smaller set has positive prior mass due to our assumption on the prior $\Pi$, one concludes that $\theta\in\kls(\Pi)$.

So fix $\theta,\theta'\geq 0$ and $x\in[0,1]$. Because $1+\sin(t)/t\in[\ell, 2]$ for all $t\geq 0$ and some $\ell>0$, we have that
\begin{align*}
    \kl(f_{\theta}, f_{\theta'}) & =\int_0^1\ln\left(\frac{f_\theta(x)}{f_{\theta'}(x)} \right) f_\theta(x) \,\mathrm dx \\
    & \leq \ln\left(\frac{1+\sin(\theta')/\theta'}{1+\sin(\theta)/\theta}\right) + \frac{1}{\ell}\int_0^1 \ln\left(\frac{1+\cos(\theta x)}{1+\cos(\theta' x)}\right)(1+\cos(\theta x)) \, \mathrm dx .
\end{align*}
By the continuity and boundedness away from $0$ and $\infty$ of the function $\theta \mapsto 1+\sin(t)/t$, the first term converges to $0$ as $\theta' \to \theta$. As for the second addendum, notice that\footnote{Here, we proceed under the assumption that $\theta>0$, the case $\theta=0$ being easily handled thanks to the strict positivity of densities with parameter lying in a neighborhood of 0.} 
\begin{align*}
    \int_0^1 \ln\left(\frac{1+\cos(\theta x)}{1+\cos(\theta' x)}\right)(1+\cos(\theta x)) \, \mathrm dx = \frac{1}{\theta} \int_0^\theta \ln\left(\frac{1+\cos(s)}{1+\cos(\theta's/\theta)}\right)(1+\cos(s))\,\mathrm ds.
\end{align*}
Now fix $s\in[0,\theta]$ and obtain
\begin{align*}
    g(\theta') & := \ln\left(\frac{1+\cos(s)}{1+\cos(\theta's/\theta)}\right)(1+\cos(s)) \\
    & = g(\theta) + g'(\theta)(\theta'-\theta) + \frac{g''(\phi)}{2} (\theta'-\theta)^2 \\
    & \leq g(\theta) + |g'(\theta)||\theta'-\theta| + \frac{|g''(\phi)| }{2} (\theta'-\theta)^2
\end{align*}
by a second order Taylor expansion with remainder around $\theta'=\theta$, where $\phi$ lies between $\theta$ and $\theta'$. Clearly $g(\theta)=0$ and
\begin{equation*}
    |g'(\theta)|= \left|\frac{s\sin(\theta s/\theta)(1+\cos(s))}{\theta(1+\cos(\theta s/\theta))}\right| = \left|\frac{s\sin(s)}{\theta}\right| \leq 1.
\end{equation*}
Moreover
\begin{align*}
    |g''(\phi)| & = \left|\frac{\cos(\phi s/\theta)}{1 + \cos(\phi s/\theta)} + \left(\frac{\sin(\phi s/\theta)}{1+\cos(\phi s/\theta)}\right)^2 \right| (1+\cos(s))\left(\frac{s}{\theta}\right)^2 \\
    & \leq 2\left[\left|\frac{\cos(\phi s/\theta)}{1 + \cos(\phi s/\theta)}\right| + \left(\frac{\sin(\phi s/\theta)}{1+\cos(\phi s/\theta)}\right)^2 \right]
\end{align*}
where it is easily shown that, for $\phi$ sufficiently close to $\theta$ (hence $\theta'$ sufficiently close to $\theta$), the two addenda in square brackets are upper-bounded by finite constants, uniformly over all $s\in[0,\theta]$. Therefore, we conclude that, for $\theta'$ sufficiently close to $\theta$,
\begin{align*}
    \frac{1}{\theta} \int_0^\theta \ln\left(\frac{1+\cos(s)}{1+\cos(\theta's/\theta)}\right)(1+\cos(s))\,\mathrm ds & \lesssim |\theta - \theta'| + (\theta - \theta')^2 \leq 2|\theta - \theta'|
\end{align*}
where the constant in the inequality may depend on $\theta$. This completes the proof.

\subsubsection*{Proof of property 3}
To prove sequential identifiability at all $\theta>0$, it is enough to observe that the model is identifiable and, as $\theta'\to\infty$, $F_{\theta'}$ does not converge weakly to $F_\theta$.

\subsubsection*{Proof of property 4}
To prove that as $\theta\to\infty$, $F_{\theta}\wc F_0$, it is enough to observe that
$$\lim_{\theta\to\infty} F_{\theta}(x) = \lim_{\theta\to\infty}\frac{x + \sin(\theta x)/\theta}{1 + \sin(\theta)/\theta} = x = F_0(x)$$
for all $x\in[0,1]$. To show instead that there exists no Hellinger limit as $\theta\to\infty$, we proceed as follows. By what we just showed, if the Hellinger limit of $f_\theta$ as $\theta\to\infty$ existed, it would be $f_0(x)\equiv 1_{[0,1]}(x)$ (because Hellinger convergence implies weak convergence and weak limits are unique). That is, recalling that Hellinger and $L^1$ convergence are equivalent, we would have
\begin{equation*}
    0 = \lim_{\theta\to\infty} \int_0^1 |f_\theta(x) - 1| \mathrm dx = \int_0^1 \lim_{\theta\to\infty}|f_\theta(x) - 1| \mathrm dx,
\end{equation*}
where the last equality comes from an application of the dominated convergence theorem to the bounded integrand $|f_\theta(x) - 1|$. Therefore, $\lim_{\theta\to\infty} f_\theta(x)$ would exist for almost every $x\in[0,1]$, which we next show not to be the case. Let $\theta_k=2\pi k$ for all $k\in\NN$ and fix $x\in[0,1]\setminus\QQ$. Because $\sin(2\pi k) = 0$ for all $k\in\NN$, we can write $f_{\theta_k}(x) = 1+\cos(2\pi k x) = 1 +\cos(2\pi \{kx\})$, where we denote by $\{r\}:= r - \lfloor r\rfloor$ the fractional part of $r\geq 0$. It is a well-known fact that, for irrational $x$, the set $\{\{kx\} : k\in\NN\}$ is dense in $[0,1]$, and because continuous functions map dense sets to dense sets, $\{1 + \cos(2\pi\{kx\}) : k\in\NN\}$ is dense in $[0,2]$. Therefore $\lim_{k\to\infty} f_{\theta_k}(x)$ does not exist for any $x\in [0,1]\setminus \QQ$, which is a set of Lebesgue measure 1. This leads to a contradiction and concludes the proof.

\subsection*{Proof of Proposition~\ref{pro:failure_walker}}%

Assumption~\Wref{walker_ass1} requires the set $\mathcal O_\theta := \{x\in[0,1] : f_\theta(x) = 0\}$ to be the same for all $\theta\geq 0$, which is clearly not true because
\begin{equation*}
    \mathcal O_\theta = \{x\in [0,1] : \exists k \in \NN_0, \,\theta x = (2k + 1)\pi\},
\end{equation*}
which depends on $\theta$.

Assumption~\Wref{walker_ass3} requires that, for any $\theta_\star\geq 0$ and sufficiently large $M>0$, there exists a function $K_M(x, \theta_\star)$ such that
\begin{equation*}
    \ln f_\theta(x) - \ln f_{\theta_\star}(x) < K_M (x, \theta_\star)
\end{equation*}
for all $\theta > M$, with
\begin{equation*}
    \lim_{M\to \infty}\int_0^1 K_M(x, \theta_\star) \, f_{\theta_\star}(x) \,\mathrm dx < 0.
\end{equation*}
From the proof of Theorem~\ref{thm:MLE_inconsistency} (see the next section), we see that, for all $x\in [0,1]\setminus \mathcal O_{\theta_\star}$, $\delta>0$ and $M>0$, there exist infinitely many $\theta>M$ such that $\ln f_\theta(x)\geq\ln(2-\delta)$, so that any candidate $K_M(x, \theta_\star)$ must satisfy
\begin{equation*}
    K_M(x, \theta_\star) \ge \ln(2-\delta) - \ln f_{\theta_\star}(x).
\end{equation*}
From the proof of Lemma~\ref{lem:upperbound_prod_likelihood} below, it emerges that $\int_0^1f_{\theta_\star}^2(x)\mathrm dx < 2$, so
\begin{align*}
    \int_0^1 f_{\theta_\star}(x)\ln\left(\frac{1}{2} f_{\theta_\star}(x)\right)\mathrm dx & \leq \int_0^1 f_{\theta_\star}(x)\left(\frac{1}{2} f_{\theta_\star}(x)-1\right)\mathrm dx \\
    & = \frac{1}{2}\int_0^1f_{\theta_\star}^2(x) - 1 \\
    & < 0,
\end{align*}
for all $\theta_\star\geq 0$, or equivalently $\int_0^1 f_{\theta_\star}(x)\ln f_{\theta_\star}(x) \, \mathrm dx<\ln 2$. Therefore, choosing $\delta>0$ small enough, we obtain
\begin{equation*}
    \int_0^1 K_M(x, \theta_\star) f_{\theta_\star}(x)\, \mathrm dx \geq \ln(2-\delta) - c_{\theta_\star} >0
\end{equation*}
for all $M>0$, a violation of assumption~\Wref{walker_ass3}.

\subsection*{Proof of Theorem~\ref{thm:MLE_inconsistency}} 

For all true data-generating parameters $\theta_\star\geq 0$, we prove that the MLE $\hat\theta_n$, if it exists, diverges to infinity a.s.-$F_{\theta_\star}^{\infty}$ in two steps as follows:
\begin{enumerate}
    \item Fix a set of $n$ distinct numbers $\{x_1, \dots, x_n\}\subset [0,1]$. We prove that, for any arbitrarily small $\delta>0$, there exists $\theta_\delta\geq 0$ such that
    \[
    \frac{1+\cos(\theta_\delta x_i)}{1+\frac{\sin\theta_\delta}{\theta_\delta}} \ge 2-\delta \quad \text{for all } i = 1,\dots, n.
    \]
    This immediately implies that, for all $n\in\mathbb N$, if $\hat\theta_n$ exists, then
    \begin{align*}
        \forall \delta>0, & \quad \prod_{i=1}^n f_{\hat\theta_n}(x_i) \geq \prod_{i=1}^n f_{\theta_\delta}(x_i) \geq (2-\delta)^n \\
        \implies & \quad \prod_{i=1}^n f_{\hat\theta_n}(x_i) \geq 2^n
    \end{align*} \label{step1}

    \item We then go back to the probabilistic setting where $X_1, \dots, X_n\iid f_{\theta_\star}$. For all $M>0$ and $\theta_\star\geq 0$, we show that
    \begin{equation*}
        F_{\theta_\star}^\infty\left(x_{1:\infty} \in[0,1]^\NN :  \max_{\theta\in[0,M]}\prod_{i=1}^n f_{\theta}(x_i) < 2^n \quad \textnormal{ultimately}\right) = 1
    \end{equation*}
    By step 1, this proves that the MLE is above $M$ for all large $n\in\NN$ a.s.-$F_{\theta_\star}^\infty$, showing that (a) it is inconsistent at $\theta_\star$, and (b) it diverges to infinity a.s.-$F_{\theta_\star}^\infty$ (because $M$ is arbitrarily large).\footnote{In step 2, any statement about the MLE $\hat\theta_n$ tacitly assumes its existence. Nevertheless, should it not exist for a certain finite sequence $y_{1:n}\in\RR^n$ (i.e., should the supremum of the product likelihood not be achieved by any $\theta\geq 0$), any event $B\in \mathscr B(\RR^\NN)$ of the form $B=\{x_{1:\infty}\in\RR^\NN : \hat \theta_n\in A\}$ (for some $A\in\mathscr B(\RR)$) should be understood to exclude all those infinite sequences $x_{1:\infty}\in\RR^\NN$ for which $x_i = y_i$ for all $i=1,..., n$.} \label{step2}
\end{enumerate}
\bigskip

\subsubsection*{Proof of step 1}

Let \(\delta > 0\) be given, let \(x_1, \dots, x_n \in [0,1]\) be distinct numbers such that $x_i/\pi$ is irrational for all $i=1,\dots,n$,\footnote{Notice that the set of all such configurations $(x_1,\dots,x_n)$ has $n$-dimensional Lebesgue measure 1, so restricting to such class of numbers is without loss of generality for our later probability statements.} and assume that $\hat\theta_n$ exists. We wish to show that there exists \(\theta > 0\) such that
\begin{equation}\label{eq:step1}
\frac{1+\cos(\theta x_i)}{1+\frac{\sin\theta}{\theta}} \ge 2-\delta \quad \text{for all } i = 1,\dots, n.
\end{equation}

As for the denominator of Equation~\eqref{eq:step1}, note that
\[
\lim_{\theta \to \infty} \frac{\sin\theta}{\theta} = 0.
\]
Thus, there exists \(\theta_0 \geq 0\) such that for all \(\theta \ge \theta_0\) we have
\[
\left|\frac{\sin\theta}{\theta}\right| < \frac{\delta}{4}.
\]
In particular, for \(\theta \ge \theta_0\),
\[
1+\frac{\sin\theta}{\theta} \le 1+\frac{\delta}{4}.
\]

As for the numerator of Equation~\eqref{eq:step1}, we now want to find $\theta\geq \theta_0$ such that
\[
1+\cos(\theta x_i) \ge 2-\frac{\delta}{2} \quad \text{for all } i=1,\dots,n,
\]
so it is enough to require
\[
\cos(\theta x_i) \ge 1-\frac{\delta}{2} \quad \text{for all } i=1,\dots,n.
\]
By the continuity of the cosine function at \(0\), there exists \(\varepsilon > 0\) (depending on \(\delta\)) such that
\[
|\phi| < \varepsilon \quad \Longrightarrow \quad \cos \phi \ge 1-\frac{\delta}{2}.
\]
Thus, if we can find $\theta\geq\theta_0$ and integers $k_1,\dots,k_n$ such that
\[
|\theta x_i - 2\pi k_i| < \varepsilon \quad \text{for all } i=1,\dots,n,
\]
then we have
\[
\cos(\theta x_i) = \cos\Bigl(\theta x_i - 2\pi k_i\Bigr) \ge 1-\frac{\delta}{2},
\]
and consequently,
\[
1+\cos(\theta x_i) \ge 2-\frac{\delta}{2}.
\]
To achieve this, we use Dirichlet's simultaneous approximation theorem:\footnote{See Corollary 1B on page 27 of \cite{schmidt2009diophantine}.} for any irrational \(x_1/2\pi, \dots, x_n/2\pi\) and for any \(M > 0\), there exist infinitely many natural numbers \(\theta \ge M\) and integers \(k_1, \dots, k_n\) such that
\[
\left|\frac{x_i}{2\pi} - \frac{k_i}{\theta}\right| < \frac{1}{\theta^{1+1/n}} \quad \text{for all } i = 1,\dots, n.
\]
Multiplying both sides of the inequality by \(2\pi \theta\), we obtain
\[
|\theta x_i - 2\pi k_i| < \frac{2\pi}{\theta^{1/n}} \quad \text{for all } i=1,\dots,n.
\]
Our goal is to have
\[
|\theta x_i - 2\pi k_i| < \varepsilon.
\]
To ensure this, it suffices to have
\[
\frac{2\pi}{\theta^{1/n}} \le \varepsilon \iff \theta \ge \left(\frac{2\pi}{\varepsilon}\right)^n.
\]
Hence, Dirichlet's simultaneous approximation theorem guarantees the existence of a natural number \(\theta\) and integers \(k_i\) such that
\[
\theta \ge \max\left\{\theta_0, \left(\frac{2\pi}{\varepsilon}\right)^n\right\}
\]
and
\begin{equation*}
     |\theta x_i - 2\pi k_i| < \varepsilon \quad \text{for all } i=1,\dots,n.
\end{equation*}

In conclusion, for the integer \(\theta\) obtained above we have \(\theta \ge \theta_0\) and so
\[
1+\frac{\sin\theta}{\theta} \le 1+\frac{\delta}{4}.
\]
Thus, for every \(i=1,\dots,n\) we obtain
\[
\frac{1+\cos(\theta x_i)}{1+\frac{\sin\theta}{\theta}} \ge \frac{2-\frac{\delta}{2}}{1+\frac{\delta}{4}}\ge 2-\delta.
\]

\begin{remark}
    As a byproduct of the above proof, we find that, for any $\delta>0$ and distinct points $(x_1, \dots, x_n)\in[0,1]^n$ in a set of $n$-dimensional full Lebesgue measure, there exist infinitely many $\theta >M$ (for any arbitrarily large $M>0$) at which the product likelihood takes a value greater than $(2-\delta)^n$. Because
    \begin{equation*}
        \max_{x\in[0,1]} f_\theta(x) = \frac{2\theta}{\theta + \sin\theta} \to 2 \quad \textnormal{as } \theta\to\infty,
    \end{equation*}
    this effectively means that, above any $M>0$, there is an infinite number of peaks of the likelihood whose height is arbitrarily close to the asymptotic maximum $2^n$.
\end{remark}

\subsubsection*{Proof of step 2}

In the second step of the proof, we show that the product likelihood, restricted to $\theta \in [0, M]$ for some $M < \infty$, cannot asymptotically attain the $2^n$ lower-bound derived in the previous step. To this end, we first establish that, for any fixed $\theta \geq 0$, the product of the likelihood values plus a small constant $\varepsilon > 0$ cannot asymptotically reach this lower-bound. The presence of this positive $\varepsilon$ then allows us, together with the equicontinuity of the likelihood function, to extend the argument uniformly over $[0, M]$ via a standard covering argument. We therefore begin with the following lemma.

\begin{lemma}\label{lem:upperbound_prod_likelihood}
    Let $X_1, \dots, X_n\iid f_{\theta_\star}$ for some $\theta_\star\geq 0$. There exists a universal constant $c< 2$ such that
    \begin{equation*}
        \EE\left[\prod_{i=1}^n (f_\theta(X_i)+ \varepsilon)\right] \leq (c+\varepsilon)^n
    \end{equation*}
    for any $\theta, \varepsilon\geq 0$.
\end{lemma}

\begin{proof}
    By the iid assumption, we get
    \begin{equation*}
        \EE\left[\prod_{i=1}^n (f_\theta(X_i)+ \varepsilon)\right] = \left(\int_0^1 f_\theta(x)f_{\theta_\star}(x) \,\mathrm dx + \varepsilon\right)^n.
    \end{equation*}
    Also
    \begin{align*}
        \int_0^1 f_\theta(x)& f_{\theta_\star}(x) \, \mathrm dx = \frac{\int_0^1 (1+\cos(\theta x))(1+\cos(\theta_\star x))\mathrm dx}{\left(1 +\frac{\sin\theta}{\theta}\right)\left( 1 +\frac{\sin\theta_\star}{\theta_\star}\right)} \\
        & = \frac{1 +\frac{\sin\theta}{\theta} + \frac{\sin\theta_\star}{\theta_\star} + \frac{1}{2}\left[\int_0^1 \cos((\theta-\theta_\star) x)\mathrm dx + \int_0^1 \cos((\theta+\theta_\star) x)\mathrm dx\right]}{\left(1 +\frac{\sin\theta}{\theta}\right)\left( 1 +\frac{\sin\theta_\star}{\theta_\star}\right)} \\
        & = \frac{1 +\frac{\sin\theta}{\theta} + \frac{\sin\theta_\star}{\theta_\star} + \frac{1}{2}\left[\frac{\sin(\theta-\theta_\star)}{\theta-\theta_\star} + \frac{\sin(\theta+\theta_\star)}{\theta+\theta_\star}\right]}{\left(1 +\frac{\sin\theta}{\theta}\right)\left( 1 +\frac{\sin\theta_\star}{\theta_\star}\right)}
    \end{align*}
    for all $\theta, \theta_\star\geq 0$. Therefore, letting $g(\theta):=\sin(\theta)/\theta$, we look for a constant $c<2$ satisfying the inequality
    \begin{equation}
        c \geq \frac{1 + 2g(\theta) + \frac{1}{2}(1+g(2\theta))}{(1+g(\theta))^2} \equiv \frac{1 + 2g(\theta) + \frac{1}{2}(1+g(\theta)\cos(\theta))}{(1+g(\theta))^2}.
    \end{equation}
    for all $\theta\geq 0$, finding that $c=1.9$ works to this end.
\end{proof}

Now fix $\varepsilon\in(0, 2-c)$ and $b\in (c+\varepsilon, 2)$, where $c<2$ is the constant obtained in Lemma~\ref{lem:upperbound_prod_likelihood}. Then, using Markov's inequality, we obtain that
\begin{equation}\label{eq:bound_single_theta}
    F_{\theta_\star}^\infty \left(x_{1:\infty}\in [0,1]^\NN\, :\,  \prod_{i=1}^n (f_\theta(x_i)+\varepsilon) \geq b^n\right) \leq \left(\frac{c+\varepsilon}{b}\right)^n \equiv e^{-dn},
\end{equation}
for all $\theta, \theta_\star\geq 0$ and $d\equiv\ln(b/(c+\varepsilon))>0$. Notice that $d$ itself, after fixing $\varepsilon$ and $b$, can be thought of as a universal constant independent of $\theta$ and $\theta_\star$. Now fix $M>0$, let $\{\theta^1, \theta^2, \dots\}$ be an enumeration of $\QQ\cap[0,M]$, and define $\mathcal T_n = \{\theta^1, \dots, \theta^n\}$. Therefore a union bound gives
\begin{equation*}
    F_{\theta_\star}^\infty \left(x_{1:\infty}\in [0,1]^\NN\, :\,  \max_{\theta\in\mathcal T_n}\prod_{i=1}^n (f_\theta(x_i)+\varepsilon) \geq b^n\right) \leq n e^{-dn}.
\end{equation*}
Because this upper-bound is summable in $n\in\NN$, the first Borel-Cantelli lemma yields $F_{\theta_\star}^\infty (\Omega_\star) = 1$, where
\begin{equation*}
    \Omega_\star := \left\{x_{1:\infty}\in[0,1]^\NN\, :\,  \max_{\theta\in\mathcal T_n}\prod_{i=1}^n (f_\theta(x_i)+\varepsilon) < b^n \quad \textnormal{ultimately}\right\}.
\end{equation*}

Now notice that, by density of $\QQ\cap[0,M]$ in $[0,M]$, for all $\delta>0$ there exists $N_\delta\in\NN$ such that, for all $n\geq N_\delta$, the maximum distance between consecutive points $\theta,\theta'\in\mathcal T_n$ is less than $\delta$. In particular, fixing $n\geq N_\delta$, for all $\theta\in[0,M]$, there exists $\theta'\in\mathcal T_n$ such that $|\theta - \theta'|< \delta$. Moreover, by verifying that $|\partial f_\theta(x)/\partial \theta|\leq L$ for all $\theta\geq 0$, $x\in[0,1]$ and some $L<\infty$, the family $\{\theta\mapsto f_\theta(x) : x\in[0,1]\}$ is easily seen to be equicontinuous. In particular, there exists $\delta>0$ such that $|\theta - \theta'|<\delta$ implies $|f_\theta(x) - f_{\theta'}(x)|<\varepsilon$ for all $x\in[0,1]$ and $\theta,\theta'\in[0,M]$ (recall that $\varepsilon>0$ has been fixed beforehand).

Now fix $x_{1:\infty}\in\Omega_\star$. Therefore, there exists $N\in\NN$ such that
\begin{equation*}
    \max_{\theta\in\mathcal T_n}\prod_{i=1}^n (f_\theta(x_i)+\varepsilon) < b^n
\end{equation*}
for all $n\geq N$. Choosing $n\geq N\lor N_\delta$ and denoting $\theta_n \in\arg\max_{\theta\in[0,M]} \prod_{i=1}^n f_\theta(x_i)$,\footnote{Notice that, by an easy application of Weierstrass's theorem, $\theta_n$ is well defined.} we have that $|\theta_n-\theta|<\delta$ for some $\theta\in\mathcal T_n$, and so
\begin{equation*}
    \prod_{i=1}^n f_{\theta_n}(x_i) \leq \prod_{i=1}^n (f_{\theta}(x_i)+\varepsilon) < b^n.
\end{equation*}
Because this holds for any $x_{1:\infty}\in\Omega_\star$, we have shown that
\begin{equation*}
    F_{\theta_\star}^\infty \left(x_{1:\infty}\in [0,1]^\NN \, :\,  \max_{\theta\in[0,M]}\prod_{i=1}^n f_\theta(x_i) < b^n \quad \textnormal{ultimately}\right) = 1.
\end{equation*}
In particular, because $b<2$, by step 1 we have
\begin{equation*}
    \left\{x_{1:\infty}\in [0,1]^\NN \, :\,  \max_{\theta\in[0,M]}\prod_{i=1}^n f_\theta(x_i) < b^n \quad \textnormal{ultimately}\right\} \subseteq \left\{x_{1:\infty}\in [0,1]^\NN \, :\hat\theta_n >M \quad \textnormal{ultimately}\right\},
\end{equation*}
showing that the MLE $\hat \theta_n$ is larger than any $M>0$ for all large $n\in\NN$, a.s.-$F_{\theta_\star}^\infty$.

\subsection*{Proof of Proposition~\ref{pro:G-VdV_consistency}}%

Lemma~\ref{lem:hellinger<euclidean} implies that the Hellinger metric entropy of the set $\Theta_n := \{f_\theta : \theta \in [0, \bar\theta_n]\}$ can be bounded as follows. Because the Hellinger metric is upper-bounded by the Euclidean metric, one can cover $\Theta_n$ with $N\leq \bar\theta_n/\delta$ Euclidean, hence Hellinger, balls of radius $\delta$. Therefore the Hellinger $\delta$-covering number of $\Theta_n$ satisfies $N(\Theta_n, d_h, \delta) \leq \bar\theta_n/\delta$. In particular, this implies that the Hellinger metric entropy (the natural log of the covering number) satisfies
\begin{equation*}
    \ln N(\Theta_n, d_h, \delta) \leq  \ln\left(\frac{\bar\theta_n}{\delta}\right).
\end{equation*}
Hence, we can use the conditions of Theorem 6.23 in \cite{ghosal2017fundamentals} to ensure that the posterior mass of any Hellinger neighborhood of $f_{\theta_\star}$, for $\theta_\star =0$, converges to 1 a.s.-$F_{\theta_\star}^{\infty}$ as $n\to\infty$: for any $\delta>0$, we require that there exists $c>0$ such that, for all large $n\in\NN$, the prior $\Pi$ satisfies $\Pi(\Theta_n^c)\leq e^{-c n}$ as long as $\ln N(\Theta_n, d_h, \delta)\leq n\delta^2$, or equivalently $\bar\theta_n\leq \delta \exp\{n\delta^2\}$.

Therefore, any prior $\Pi$ on $[0,\infty)$ that, for all $\delta>0$, admits $c>0$ such that
\begin{equation*}
    \Pi\left(\left[\delta \exp\{n\delta^2\}, \, \infty\right)\right) \leq e^{-cn} \quad \textnormal{for all large } n\in\NN,
\end{equation*}
will ensure Hellinger consistency at the uniform distribution. We now show that, for any $\varphi :(0, \infty) \to (0,\infty)$ such that $\lim_{t\to\infty} \varphi(t)/t = \infty$, a prior $\Pi$ with
\begin{equation*}
    \Pi([\theta, \infty)) \leq e^{-\varphi(\ln \theta)}, \quad 
    \textnormal{for all large } \theta>0
\end{equation*}
satisfies the previous condition. Indeed, for all $\delta>0$, the properties of $\varphi$ imply that $\varphi(\ln\delta + n\delta^2) \geq \ln\delta + n\delta^2 \geq c n$ for all $n\in\NN$ large and some small $c>0$. Therefore,
\begin{equation*}
    \Pi\left(\left[\delta \exp\{n\delta^2\}, \, \infty\right)\right) \leq \exp\left\{-\varphi(\ln \delta + n\delta^2)\right\} \leq e^{-cn} \quad \textnormal{for all large } n\in\NN,
\end{equation*}
as desired.

\subsection*{Proof of Proposition~\ref{pro:consistency_walker2004}}

For any $\delta>0$, because the sequence $\theta_k:=(1+2k)\delta$, $k\in\NN_0$, is a $\delta$-cover of $[0,\infty)$ (that is, for all $\theta\geq 0$, there exists $k\in\NN_0$ such that $|\theta - \theta_k|\leq \delta$), the sequence $f_{\theta_k}$ is a $\delta$-cover, in the Hellinger sense, of $\{f_\theta : \theta\in[0,\infty)\}$ (thanks to Lemma~\ref{lem:hellinger<euclidean}). Therefore, defining
\begin{equation*}
    A_k = \{f_\theta : |\theta - \theta_k|\leq \delta\} \equiv \left\{f_\theta : \theta \in \left[2k\delta, \, 2k\delta + 2\delta\right]\right\}
\end{equation*}
for all $k\in\NN_0$, this shows that $A_k \subseteq A_k^\star:=\{f_\theta : d_h(f_\theta, f_{\theta_k}) \leq \delta\}$. Thus, we can use Theorem 4 in \cite{walker2004squarerootsum} to design a prior $\Pi$ such that
\begin{equation*}
    \sum_{k\in\NN_0}\sqrt{\Pi(A_k)} \equiv \sum_{k\in\NN_0}\sqrt{\Pi\left(\left[2k\delta, \, 2k\delta + 2\delta\right]\right)} < \infty
\end{equation*}
and conclude that the posterior is Hellinger consistent at $\theta_\star = 0$. In particular, we have assumed that $\Pi$ admits a density $\pi(\theta)$ on $[0,\infty)$ that, for all large $\theta>0$, is decreasing and satisfies
\begin{equation*}
    \pi(\theta) \lesssim  \frac{1}{\theta^2 (\ln \theta)^{2 + \beta}},
\end{equation*}
for some $\beta>0$. Then, for $k_0\in\NN_0$ large enough,
\begin{align*}
    \sum_{k=k_0}^\infty\sqrt{\Pi\left(\left[2k\delta, \, 2k\delta + 2\delta\right]\right)} & \lesssim \sum_{k=k_0}^\infty\sqrt{\frac{2\delta}{(2k\delta)^2 (\ln (2k\delta))^{2 + \beta}}} \\
    & = \frac{1}{\sqrt{2\delta}} \sum_{k=k_0}^\infty \frac{1}{k\ln(2k\delta)^{1+\beta/2}} < \infty
\end{align*}
for all $\delta>0$, showing that $\Pi$ yields strong consistency at $\theta_\star = 0$.

\subsection*{Proof of Theorem~\ref{thm:our_consistency}}%

For this proof, denote by $\hat\theta_n$ any MLE restricted to the sieve $[0,M_n] \cap A_\varepsilon^c$, where $M_n$ is a positive sequence determined throughout the proof, and $A_\varepsilon := \{\theta\geq 0 : d_h(f_\theta, f_0)<\varepsilon\}$ for some small $\varepsilon>0$. Then, for $\delta>0$ chosen small enough, for all $\theta\in A_\varepsilon^c$ we have
\begin{align*}
    F_{0}^\infty\left( \prod_{i=1}^n\left(f_\theta(X_i)^{1/2} +\delta\right) \geq e^{-nb/2}\right) & \leq e^{nb/2} \left(\int_0^1\sqrt{f_\theta(x)f_{\theta_\star}(x)}\,\mathrm dx + \delta\right)^n \\
    & =e^{nb/2} \left(1- d_h^2(f_\theta, f_{\theta_\star})/2 + \delta\right)^n \\
    & \leq e^{nb/2} (1-\varepsilon^2/2 + \delta)^n \\
    & \leq e^{nb/2} (1-\varepsilon^2/4)^n \\
    & \leq e^{nb/2} e^{-n\varepsilon^2/4},
\end{align*}
which, choosing $b>0$ small enough, is smaller than $e^{-nC_\varepsilon}$ for some $C_\varepsilon>0$. Let $M_n = e^{nc}$ for some $c<C_\varepsilon$ and let $\eta>0$ be such that $|\theta-\theta'|<\eta \implies \max_{x\in[0,1]}|\sqrt{f_\theta(x)} - \sqrt{f_{\theta'}(x)}| < \delta$.\footnote{Recall that $\{\theta \mapsto f_\theta(x): x\in[0,1]\}$ is uniformly Lipschitz, therefore equicontinuous.} Therefore, constructing an $\eta$-cover $\{\theta^1,\theta^2,\dots\}$ of $[0,M_n]\cap A_\varepsilon^c$ of cardinality at most $M_n/\eta$, a union bound gives
\begin{align*}
    F_{0}^\infty\left( \max_{\theta \in  [0,M_n] \cap A_\varepsilon^c} \prod_{i=1}^nf_\theta^{1/2}(X_i) \geq e^{-nb/2}\right) & \leq \frac{M_n}{\eta} F_{0}^\infty\left( \prod_{i=1}^n\left(f_{\theta^1}^{1/2}(X_i)+\delta\right) \geq e^{-nb/2}\right)\\
    & \leq \frac{1}{\eta}e^{-(C_\varepsilon - c)n}.
\end{align*}
The above upper-bound is summable in $n\in\NN$ and therefore we obtain that
\begin{equation}\label{eq:sieve_MLE_consistent}
    \prod_{i=1}^n\left(\frac{f_{\hat\theta_n}(X_i)}{f_{\theta_\star}(X_i)}\right)^{1/2} \equiv \prod_{i=1}^n f_{\hat\theta_n}^{1/2}(X_i) < e^{-nb/2}
\end{equation}
ultimately a.s.-$F_{0}^\infty$.

Now write
\begin{equation}\label{eq:complement_decomposition}
    \Pi(A_\varepsilon^c\mid X_{1:n}) = \Pi(A_\varepsilon^c \cap [0,e^{cn}]\mid X_{1:n}) + \Pi(A_\varepsilon^c \cap (e^{cn},\infty)\mid X_{1:n}).
\end{equation}
Using a line of reasoning similar to \cite{walker2001}, rewrite the first addendum as follows:
\begin{equation}\label{eq:first_addendum}
    \Pi(A_\varepsilon^c \cap [0,e^{cn}] \mid X_{1:n}) = \frac{\int_{A_\varepsilon^c \cap [0,e^{cn}]}\prod_{i=1}^n\frac{f_\theta(X_i)}{f_{\theta_\star}(X_i)} \,\Pi(\mathrm d\theta)}{\int_{\Theta}\prod_{i=1}^n\frac{f_\theta(X_i)}{f_{\theta_\star}(X_i)} \,\Pi(\mathrm d\theta)}.
\end{equation}
A standard result \citep[see, e.g.,][Lemma 4]{barron1999} ensures that, as long as $\theta_\star \in \kls(\Pi)$, the denominator satisfies 
\begin{equation}\label{eq:our_denominator}
    \int_{\Theta}\prod_{i=1}^n\frac{f_\theta(X_i)}{f_{\theta_\star}(X_i)} \,\Pi(\mathrm d\theta) > e^{-dn} 
\end{equation}
ultimately a.s.-$F_{\theta_\star}^\infty$ for all $d>0$. For the numerator, we can write
\begin{align}\label{eq:walker_hjort_numerator}
    \int_{A_\varepsilon^c \cap [0,e^{cn}]}& \prod_{i=1}^n\frac{f_\theta(X_i)}{f_{\theta_\star}(X_i)} \,\Pi(\mathrm d\theta) \nonumber\\
    & \leq \prod_{i=1}^n\left(\frac{f_{\hat\theta_n}(X_i)}{f_{\theta_\star}(X_i)}\right)^{1/2} \int_{A_\varepsilon^c \cap [0,e^{cn}]}\prod_{i=1}^n\left(\frac{f_\theta(X_i)}{f_{\theta_\star}(X_i)}\right)^{1/2} \Pi(\mathrm d\theta).
\end{align}
For the second factor in Equation~\eqref{eq:walker_hjort_numerator}, one obtains
\begin{align*}
    \mathbb E_{\theta_\star}\left[\int_{A_\varepsilon^c \cap [0,e^{cn}]}\prod_{i=1}^n\left(\frac{f_\theta(X_i)}{f_{\theta_\star}(X_i)}\right)^{1/2} \Pi(\mathrm d\theta)\right] &\leq \int_{A_\varepsilon^c}\prod_{i=1}^n\EE_{\theta_\star}\left[\left(\frac{f_\theta(X_i)}{f_{\theta_\star}(X_i)}\right)^{1/2}\right] \Pi(\mathrm d\theta) \\
    & \leq \Pi(A_\varepsilon^c) (1-\varepsilon^2/2)^n \\
    & \leq \Pi(A_\varepsilon^c) e^{-n\varepsilon^2/2},
\end{align*}
so that, by Markov's inequality and the first Borel–Cantelli lemma,
\begin{equation}\label{eq:walkerhjort_numerator}
    \int_{A_\varepsilon^c \cap [0,e^{cn}]}\prod_{i=1}^n\left(\frac{f_\theta(X_i)}{f_{\theta_\star}(X_i)}\right)^{1/2} \Pi(\mathrm d\theta) < e^{-n\varepsilon^2/4}
\end{equation}
ultimately a.s.-$F_{\theta_\star}^\infty$. Thus, putting together Equations~\eqref{eq:sieve_MLE_consistent}-\eqref{eq:walkerhjort_numerator} and choosing $d>0$ small enough, we obtain
\begin{equation*}
    \lim_{n\to\infty}\Pi(A_\varepsilon^c \cap [0,e^{cn}]\mid X_{1:n}) =0
\end{equation*}
a.s.-$F_{\theta_\star}^\infty$.

As for the second addendum in Equation~\eqref{eq:complement_decomposition}, write
\begin{equation}\label{eq:last_equation}
    \Pi(A_\varepsilon^c \cap (e^{cn},\infty)\mid X_{1:n}) \leq \Pi( (e^{cn},\infty)\mid X_{1:n}) = \frac{\int_{e^{cn}}^\infty \prod_{i=1}^n f_\theta(X_i) \pi(\theta) \mathrm d\theta}{\int_0^\infty\prod_{i=1}^n f_\theta(X_i) \pi(\theta) \mathrm d\theta},
\end{equation}
and notice that, because the true density $f_0$ is equal to 1 on $[0,1]$, we can interpret $\prod_{i=1}^n f_\theta(X_i)$ as the likelihood ratio. For the denominator, we once again invoke the KL support condition to satisfy Equation~\eqref{eq:our_denominator}. As for the numerator, for all large $n\in\NN$ we have
\begin{align*}
    \mathbb E\left[\int_{e^{cn}}^\infty \prod_{i=1}^n f_\theta(X_i) \pi(\theta) \mathrm d\theta \right] & = \int_{e^{cn}}^\infty \mathbb E\left[\prod_{i=1}^n f_\theta(X_i)\right] \pi(\theta) \mathrm d\theta\\
    & = \int_{e^{cn}}^\infty \pi(\theta) \mathrm d\theta \\
    & \lesssim \int_{e^{cn}}^\infty \theta^{-(1+\alpha)} \mathrm d\theta \\
    & =\frac{1}{\alpha}e^{-\alpha cn},
\end{align*}
where the first equality follows from an application of Fubini's theorem and the second one from the assumption that $X_1,\dots,X_n\iid \textnormal{Unif}(0,1)$. Therefore, Markov's inequality and the first Borel-Cantelli lemma imply that the numerator of the right-hand side of Equation~\eqref{eq:last_equation} is smaller than $e^{-\alpha cn/2}$ ultimately a.s.-$F_0^\infty$. Finally, choosing $d>0$ small enough, we conclude that
\begin{equation*}
    \lim_{n\to\infty}\Pi(A_\varepsilon^c \cap (e^{cn},\infty)\mid X_{1:n}) =0
\end{equation*}
a.s.-$F_{\theta_\star}^\infty$.

\clearpage
\bibliography{arxiv.bib}
\bibliographystyle{apalike}

\end{document}